\documentclass[12pt]{amsart}

\usepackage{amssymb,graphicx}

\def\Q{\mathbb{Q}}
\def\Z{\mathbb{Z}}

\def\C{\mathbb{C}}

\def\a{\alpha}

\def\G{\Gamma}

\def\<{\langle}
\def\>{\rangle}
\def\ol{\overline}

\def\tr{\mathrm{Tr}}

\newcounter{app}\newcounter{subapp}
\renewcommand{\appendix}[1]{\setcounter{subapp}{0}\addtocounter{app}{1} \section*{Appendix \Alph{app}: #1}}
\newcommand{\subappendix}[1]{\addtocounter{subapp}{1} \subsection*{\Alph{app}.\arabic{subapp}: #1}}

\newtheorem{theorem}{Theorem}
\newtheorem{lemma}[theorem]{Lemma}
\newtheorem{prop}[theorem]{Proposition}
\newtheorem{cor}[theorem]{Corollary}

\newenvironment{pf}{\par\medskip\noindent\textit{Proof.}~}{\hfill $\square$\par\medskip}

\begin{document}
\title[Generalised triangle groups of type (2,3,2)]{Generalised triangle groups of type (2,3,2) with no cyclic essential representations}
\author{James Howie}
\address{ James Howie\\
Department of Mathematics and Maxwell Institute for Mathematical Sciences\\
Heriot--Watt University\\
Edinburgh EH14 4AS }
\email{J.Howie@hw.ac.uk}
\author{Olexandr Konovalov}
\address{Olexandr Konovalov\\
School of Computer Science\\
University of St Andrews\\ 
North Haugh\\ St Andrews KY16 9SX}
\email{obk1@st-andrews.ac.uk}
\date{\today}
\maketitle

\section{Background}

A
{\em generalised triangle group} is a group $G$ given by a presentation
of the form
$$\< x,y|x^p=y^q=W(x,y)^r=1\>,$$
where $p,q,r>1$ are integers and $W$ is a cyclically reduced word in the
free product $$\< x,y|x^p=y^q=1\>\cong\Z_p\ast\Z_q$$ of even syllable length
$2\ell>0$.  We say that $G$ is of type $(p,q,r)$ and has {\em length parameter} $\ell$.

\medskip
A conjecture of Rosenberger \cite{Ros} says that this class of groups satisfies
a Tits alternative: either $G$ is virtually soluble, or $G$ contains
a non-abelian free subgroup.  Known results to date (see for example the summary
in \cite{How}) reduce the conjecture
to the three cases $(p,q,r)\in \{(2,3,2),(2,4,2),(2,5,2)\}$.

An {\em essential representation} from $G$ to a group $H$ is a homomorphism
$\phi:G\to H$ such that $\phi(x)$, $\phi(y)$ and $\phi(W)$ have orders $p,q,r$ respectively.
In this paper we concentrate on the case where $(p,q,r)=(2,3,2)$ and $G$ does not admit
an essential representation onto a cyclic group. For a generalised triangle group of type $(2,3,2)$, an essential representation onto a cyclic group exists
if and only if the exponent sums of $x$ and $y$ in $W$ are respectively odd and divisible by $3$. Hence the absence of such a representation can arise in one of two ways: either
$G$ has even length parameter, or it has odd length parameter and the exponent-sum of $y$ in $W$ is not divisible by $3$.

\medskip When $G$ has even length parameter,
the conjecture was proved 
 modulo six outstanding cases in \cite{arxiv}. Jack Button (private communication) has since been able to use his largeness-testing software \cite{Button} to prove the conjecture
 for four of these six.  In Appendix A below we include logs of GAP and MAPLE sessions which confirm Buttons largeness results for these four groups, as well
 as verifying the Rosenberger conjecture for one of the remaining two groups.  The only
 remaining open case of the six exceptions in \cite{arxiv} is
\begin{equation}\label{evenexception}
\<x,y|x^2=y^3=((xy)^4(xy^2)^3(xy)^2xy^2)^2=1\>.
\end{equation}
 
\medskip   The argument in \cite{arxiv} involved finding an upper bound
for $\ell$ ($\ell\le 40$) in any putative counterexample,
and systematically analysing the finite (but large)
set of words no longer than that bound.

In the current note we concentrate mainly on the case where $(p,q,r)=(2,3,2)$, $\ell$
is odd and the exponent-sum of $y$ in $W$ is not divisible by $3$.
Again we achieve this by finding a bound on $\ell$ (in this case $\ell\le 49$)
and analysing (using parallel computations with GAP \cite{GAP}) the
resulting finite set of words.

We show that, with at most one exception (up to isomorphism), the Rosenberger
Conjecture holds for all such groups.  The  exception is the
group
\begin{equation}\label{oddexception}
\< x,y|x^2=y^3=((xy)^4(xy^2)^3(xy)^2(xy^2)^2)^2=1\>.
\end{equation}
We have not been able to prove the conjecture for this group, but there is
no evidence to suggest that it fails.

Unfortunately, this approach does not apply to the case where $(p,q,r)=(2,3,2)$ and $G$ admits an essential cyclic representation,
in other words
$\ell$ is odd and the exponent-sum of $y$ in $W$ is divisible by $3$: we cannot
find a theoretical upper bound for $\ell$ in that case.

Summarizing the above discussion, we have

\begin{theorem}\label{main}
Let $G$ be a generalised triangle group of type $(2,3,2)$ that does not admit an essential
representation onto a cyclic group.  If $G$ is a counterexample to the Rosenberger conjecture, then it is isomorphic to one of the following:
$$\<x,y|x^2=y^3=((xy)^4(xy^2)^3(xy)^2xy^2)^2=1\>,$$
$$\< x,y|x^2=y^3=((xy)^4(xy^2)^3(xy)^2(xy^2)^2)^2=1\>.$$
\end{theorem}

In \S \ref{even} below we treat the even length case. We outline the proof that $5$ of the $6$ exceptional cases from \cite[Table 2]{arxiv} contain non-abelian free subgroups.  The full computational details are postponed to Appendix A.  In the remainder of the paper we concentrate on the odd-length case.

In \S \ref{tracepoly}  we recall the definition of the {\em trace polynomial}
of $W$, and use it to obtain strong restrictions on the form of $W$ in any
putative counterexample to the Rosenberger Conjecture (Theorem \ref{Calculation}).
In \S \ref{search} we describe a computer search that uses these restrictions to
refine the set of possible counterexamples (up to a simple equivalence
relation) to a set of size $31$.  In \S \ref{sc} we use small-cancellation theory
to eliminate $16$ of these groups as potential counterexamples. A more complicated application of small-cancellation
techniques is used to eliminate one further group -- stated as Theorem \ref{free13} in \S \ref{asc} and proved in Appendix C.  Finally in \S \ref{ah} we use a variety of
further methods to eliminate $13$ more groups, leaving us with the single exceptional
case shown above. 

 Our work involves a certain amount of computation.
The graphs in the proof of Theorem \ref{Calculation} were produced using MAPLE, which was also used for the Gr\"{o}bner basis calculations in Appendix A.
All other computations were done using GAP \cite{GAP}.  Logs of GAP and MAPLE
sessions
verifying most of our computational results are included in  Appendices A (for the even length case) and B (for the odd length case).  Appendix C contains a detailed proof using pictures
of the existence of free subgroups in one of the groups concerned.

\section{Even length}\label{even}
In this section we outline the investigations into the case where $\ell$ is even.  Detailed calculations are given in Appendix A.  As mentioned in the Introduction, \cite{arxiv} deals with all such groups with $6$ exceptions.
Button has used his largeness test \cite{Button} to deal with $4$  of the $6$ exceptions
and we confirm this by GAP calculations in Appendix A.  We are currently
unable to handle the fifth exceptional case (\ref{evenexception}) above.  However
we are able to confirm that the sixth exception (Group 7b in \cite[Table 2]{arxiv})
has non-abelian free subgroups.

\begin{theorem}\label{G7b}
The group
$$G=\< x,y|x^2=y^3=((xy)^3(xy^2)^2xy(xy^2)^2xyxy^2)^2=1\>$$
contains non-abelian free subgroups.
\end{theorem}

\begin{proof}
The commutator subgroup $[G,G]$ of $G$ has index $6$ and a presentation
with two generators $u,v$ and six relators:
$$W_1^2=W_1(u,v)^2=(v^2uv^{-1}u^2)^2,$$
$$W_2^2=W_2(u,v)^2=(v^2u^2v^{-1}u)^2,$$
$$W_3^2=W_3(u,v)^2=(v^2u^{-3}vu^{-1})^2,$$
$$W_4^2=W_4(u,v)^2=(v^2u^{-1}vu^{-3})^2,$$
$$W_5^2=W_5(u,v)^2=(v^2u^{-1}vu^{-1}vu)^2,$$
$$W_6^2=W_6(u,v)^2=(v^2uvu^{-1}vu^{-1})^2.$$
A Gr\"{o}bner basis calculation shows that $[G,G]$ has a representation
$\sigma:[G,G]\to PSL(2,\C)$, $u\mapsto M$, $u\mapsto N$, where
$$M=\left(\begin{array}{cc} 0 & -1\\ 1 & a\end{array}\right),~~~
N=\left(\begin{array}{cc} -z & 1 + bz + z^2\\ -1 & z+b\end{array}\right),$$
$z$ is an algebraic integer of degree $48$ and $a=Trace(M)$ and
$b=Trace(N)$ are roots in $\Q[z]$ of the irreducible
polynomial $$p(X):=X^{12}-3X^{10}-9X^8+42X^6-48X^4+15X^2+1.$$

Now a representation $\rho:\<u,v\>\to PSL(2,\C)$ is reducible only if
$$Trace(\rho(u))^2+Trace(\rho(v))^2+Trace(\rho(uv))^2 + Trace(\rho(u))Trace\rho(v))Trace(\rho(uv))=4.$$
A further Gr\"{o}bner basis calculation shows that
$$a^2+b^2+c^2+abc\neq 4~~\mathrm{in}~~\Q[z],$$ where $c=Trace(MN)$.
Hence the above representation $\sigma$ is irreducible.  Moreover,
since $p$ is irreducible of degree $12$, none of $a,b,c$ can be the
trace of a matrix of order less than $25$ in $SL(2,\C)$, so the image of $\sigma$ in $PSL(2,\C)$ cannot be a finite or dihedral subgroup.  In particular $\sigma([G,G])$, and hence $G$, contains a non-abelian free subgroup.

Logs of sessions in GAP and MAPLE carrying out the calculations
listed above are given in Appendix A below.
\end{proof}

\section{Trace polynomial}\label{tracepoly}

The standard tool for investigating generalised triangle groups is the trace polynomial, defined as follows.
Construct matrices $X,Y\in \mathrm{SL}_2(\C[\lambda])$ such that the traces of $X,Y,XY$ are
respectively equal to $2\cos(\pi/p)$, $2\cos(\pi/q)$ and $\lambda$.  The trace polynomial
$\tau_W(\lambda)$ is defined to be the trace of $W(X,Y)$.
It is not difficult to verify that such a choice of matrices is always
possible, and that $\tau_W$ does not depend on the choice.
Moreover, $\tau_W(\lambda)$ is a polynomial of degree $k$ in
$\lambda$, and the combinatorics of the word
$W$ determine the coefficients of $\tau_W$ in a well-understood way.  Details
can be found, for example, in \cite{HW}.

If $\alpha\in\C$ is a root of $\tau_W(\lambda)-2\cos(\pi/r)$, then we can obtain an
{\em essential representation} of $G$ (one which maps $x,y,w$ to elements of orders $p,q,r$
respectively) into $\mathrm{PSL}_2(\C)$ by putting $\lambda=\alpha$ in $X,Y$.
Provided the image of this representation is non-elementary, it will contain a non-abelian free
subgroup by Tits' results, and hence so will $G$.  Thus one is reduced to consideration
of words $W$ such that the roots of $\tau_W(\lambda)-2\cos(\pi/r)$ all lie within a certain small
finite set.

In the case of interest here, $(p,q,r)=(2,3,2)$, $X,Y$ have traces $0,1$ respectively,
and we are interested in the roots of $\tau_W(\lambda)$ itself.  It can moreover be readily
shown that in this case $\tau_W(\lambda)$ is a monic polynomial with integer coefficients,
and is odd or even depending on the parity of the length parameter $\ell$.

In our case, $\ell$ is odd, so $\tau_W$
is odd.
Moreover, the roots of $\tau_W$ that correspond to
elementary representations are $0$ (corresponding to a representation onto $S_3$),
$\pm 1$ ($A_4$), $\pm \sqrt{2}$
($S_4$), $\pm \sqrt{3}$
($\Z_6$ or a parabolic subgroup of the form
$\Z^2\rtimes\Z_6$)
and $\frac{\pm 1\pm\sqrt{5}}2$ ($A_5$).

Since $\tau_W$ is odd, the negative of any root is also a root of the same multiplicity.
Since $\tau_W$ has integer coefficients, $\frac{1+\sqrt{5}}2$ is a root if and only
if $\frac{1-\sqrt{5}}2$ is a root, in which case they have the same multiplicity.  It follows
therefore that $G$ admits a non-abelian free subgroup except possibly in the cases where
\begin{equation}\label{tp}
\tau_W(\lambda)=\lambda^a(\lambda^2-1)^b(\lambda^2-2)^c(\lambda^2-3)^d(\lambda^4-3\lambda^2+1)^e
\end{equation}
for some non-negative integers $a,b,c,d,e$ with $a$ odd.

\medskip
We next note one further property of the trace polynomial.

\begin{lemma}\label{abs2}
For any real number $z$ with $|z|\le\sqrt{3}$, $|\tau_W(z)|\le 2$.
\end{lemma}

\begin{pf}
The bounds on $z$ mean that it is possible to find matrices $X,Y\in SU(2)$ with
$\tr(X)=0$, $\tr(Y)=1$ and $\tr(XY)=z$.  Hence $\tau_W(z)=\tr(W(X,Y))\in [-2,2]$
since $W(X,Y)\in SU(2)$.
\end{pf}

\medskip
Recall our assumption that the exponent-sum of $y$ in $W$
is not divisible by $3$. It follows that there can be no essential representations from $G$ onto $\Z_6$ or
onto $A_4$, and hence $b=d=0$. 
Thus (\ref{tp}) becomes in this case
$$\tau_W(\lambda)=\lambda^a(\lambda^2-2)^c(\lambda^4-3\lambda^2+1)^e.$$

In particular, $\tau_W(\sqrt{3})=3^{a/2}$, so we also have $a\le 1$ by Lemma \ref{abs2}.
Since $a$ is odd this actually means $a=1$.  So the precise form of the trace polynomial
in this case depends only on the two integers $c,e$. We find
bounds on $c,e$ as follows.

\begin{theorem}\label{Calculation}
Let $G$ be a generalised triangle group of type $(2,3,2)$ with odd length parameter
$\ell$ such that the exponent-sum of $y$ in $W(x,y)$ is not divisible by $3$.
Suppose that $G$ does not contain a non-abelian free subgroup.
Then
$$\tau_W(\lambda)=\lambda(\lambda^2-2)^c(\lambda^4-3\lambda^2+1)^e$$
with
$$c\le 4~~\mathrm{and}~~e\le 2c+2.$$
In particular $W$ has length parameter
$$\ell=1+2c+4e\le 49.$$
\end{theorem}

\begin{pf}
The form of $\tau_W(\lambda)$ is as stated in the theorem, from the  preceding remarks.
We must check the inequalities relating $c,e$.

Let $f(\lambda):=(\lambda^4-3\lambda^2+1)$ and $g(\lambda):=(\lambda^2-2)f(\lambda)^2$.  Let $\lambda_0=0.1$ and $\lambda_1:=1.15$.  Then calculations show that
$$|f(\lambda_1)|~\sim~1.22>1>0.97~\sim~|f(\lambda_0)|,~~
|g(\lambda_0)|~\sim~1.87>1~~\mathrm{and}~~|g(\lambda_1)|~\sim~1.01>1.$$

\medskip
\begin{center}
\begin{figure}
\includegraphics[height=55mm]{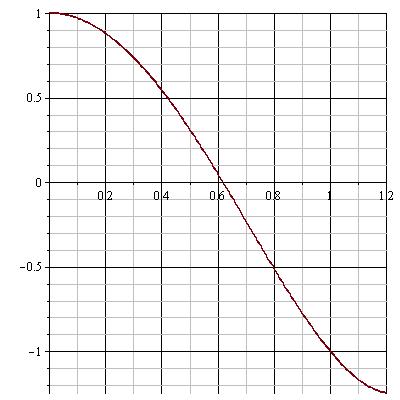}
\hspace{10mm}
\includegraphics[height=55mm]{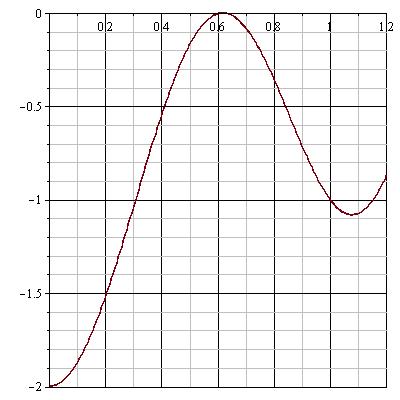}\\
\hspace*{10mm}
$y=f(x)=(x^4-3x+1)$
\hspace{15mm}
$y=g(x)=(x^2-2)(x^4-3x+1)^2$
\caption{Graphs of $f$ and $g$}
\label{maplegraphs}
\end{figure}
\end{center}

\medskip
Now put $$\sigma_0:=\lambda_0(\lambda_0^2-2)^5(\lambda_0^4-3\lambda_0^2+1)^{12}$$  and 
$$\sigma_1:=\lambda_1(\lambda_1^4-3\lambda_1^2+1)^3.$$
Then further
calculations show that $|\sigma_0|~\sim~2.17>2$ and $|\sigma_1|~\sim~2.08>2$.

\medskip
\begin{center}
\begin{figure}
\includegraphics[width=80mm,height=55mm]{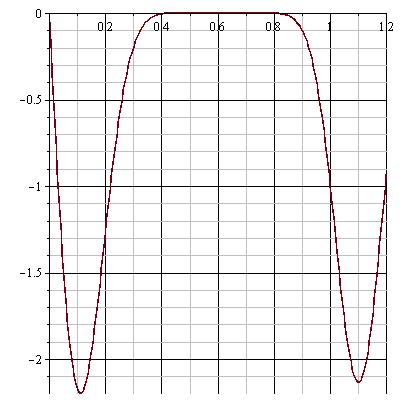}\\
\caption{Graph of $\sigma_0(x)=x(x^2-2)^5(x^4-3x+1)^{12}$}
\label{maple2}
\end{figure}
\end{center}

\medskip
Hence if $\tau_W$ has the above form with $e\ge 2c+3$, we have
$$|\tau_W(\lambda_1)|=|\sigma_1|\times|g(\lambda_1)|^c\times|f(\lambda_1)|^{e-2c-3}>2,$$
a contradiction to Lemma \ref{abs2}, while if $c\ge 5$ and $e\le 2c+2$ then
 $$|\tau_W(\lambda_0)|=\frac{|\sigma_0|\times|g(\lambda_0)|^{c-5}}{|f(\lambda_0)|^{2c+2-e}}>2,$$
also a contradiction to Lemma \ref{abs2}.
\end{pf}

\section{Searching}\label{search}

The bounds obtained on $c$ and $e$ (and hence also $\ell$)
in Theorem \ref{Calculation} show that only finitely many words have
trace polynomial of the appropriate form.  In principle it is possible to
determine all such words by a computer search.  In order to make such a search
more practical, we first prove a result that restricts the form of words that
can arise.  We then describe a search algorithm that will identify all
words with the correct form of trace polynomial, and finally give the
results of this search.

\subsection{Restricting the form of $W$} 

 Up to cyclic permutation, we may write a word $W$ in the form
$(xy)^\ell$, $(xy^2)^\ell$ or
$$W=(xy)^{b(1)}(xy^2)^{b(2)}\cdots (xy)^{b(t-1)}(xy^2)^{b(t)}$$
with $b(j)\ge 1$ for $j=1,\dots,t$.  Each maximal subword of the form $(xy)^{b(j)}$
or $(xy^2)^{b(j)}$ will be called a {\em block}, and the index $b(j)$ will be called
the {\em length} of the block.  Note that the number of blocks is either $1$
or an even number.

\begin{lemma}\label{blocks}
If $W$ has the form $(xy)^\ell$ or $(xy^2)^\ell$ and $\tau_W(\lambda)$
has the form $$\lambda(\lambda^2-2)^c(\lambda^4-3\lambda^2+1)^e$$ then $\ell=1$
and $c=e=0$.  If $\tau_W(\lambda)$
has the form $$\lambda(\lambda^2-2)^c(\lambda^4-3\lambda^2+1)^e$$ with $c+e>0$ then
$W$ has $2e+2$ blocks, of which $c+e$ have length greater than $1$.  In particular
$e\ge c-2$.
\end{lemma}

\begin{pf}
If $W=(xy)^\ell$ then $2\cos(\pi/2\ell)$ is a root of $\tau_W$, so $\tau_W$
can have the given form only for $\ell\le 2$.  But also $\ell=\mathrm{degree}(\tau_W)$ is odd, so the only possibility is that $\ell=1$ and $c=e=0$ as claimed.

\medskip
Now suppose that $c+e>0$.  We analyse the coefficients
of $\lambda^{\ell-2}$ and $\lambda^{\ell-4}$ in $\tau_W(\lambda)$, using the formulae in
\cite{HW}.  Write $$W=xy^{\alpha(1)}\cdots xy^{\alpha(\ell)}$$
with $\alpha(j)\in\{1,2\}$ for each $j$.  Define $\beta(j)=-exp(i\pi (\alpha(j+1)-\alpha(j))/3)$
(where $j$ is interpreted modulo $\ell$). Then the coefficient $B_1$ of $\lambda^{\ell-2}$
is the sum of all the $\beta(j)$: $B_1=\sum_j\beta(j)$, while the coefficient $B_2$ of
$\lambda^{\ell-4}$ is the sum of products $\beta(j)\beta(k)$, taken over all unordered pairs
$\{j,k\}$ with $k\notin\{j-1,j,j+1\} \mod \ell$.

As noted in \cite{HW}, these coefficients satisfy the following equation:
$$B_1^2=2B_2+\sum_j\beta(j)^2 +2\sum_j\beta(j)\beta(j+1).$$

We can also compute the coefficients directly from the formula
$$\tau_W(\lambda)=\lambda(\lambda^2-2)^c(\lambda^4-3\lambda^2+1)^e,$$
giving
$$B_1=-2c-3e;\qquad\qquad B_2= 4\frac{c(c-1)}2 + 9\frac{e(e-1)}2 + 6ce + e.$$

Putting these equations together gives

$$\sum_j\beta(j)^2 +2\sum_j\beta(j)\beta(j+1) = 4c + 7e.$$

Now the equation $B_1=-2c-3e$ means that, of the $\ell=1+2c+4e$ values of $j$, precisely
$2c+2e-1$ give $\beta(j)=-1$, while $e+1$ give $\beta(j)=-exp(i\pi/3)$ (corresponding to points
in $W$ where $xy$ is followed by $xy^2$) and $e+1$ give $\beta(j)=-exp(-i\pi/3)$
(places where $xy^2$ is followed by $xy$).  It then follows that
$W$ can (up to cyclic conjugacy) be subdivided into $2e+2$ blocks, which are alternatingly powers
of $xy$ and of $xy^2$.

\medskip
It also follows that
$$\sum_j\beta(j)^2= (2c+2e-1) - (e+1) = 2c + e - 2,$$
and hence
$$\sum_j\beta(j)\beta(j+1) = c + 3e + 1.$$

Now consider the above subdivision of $W$ into blocks.  If a block is a power $(xy)^s$ or
$(xy^2)^s$ with $s>1$, then the start of the block reads $xy^2xyxy$ or $xyxy^2xy^2$, giving
a value of $\beta(j)\beta(j+1)$ of $exp(-i\pi/3)$ or $exp(i\pi/3)$ respectively,
while the end of the block reads $xyxyxy^2$ or $xy^2xy^2xy$, giving
$\beta(j)\beta(j+1)=exp(i\pi/3)$ or $exp(-i\pi/3)$ respectively.  Thus each block
$(xy)^s$ or $(xy^2)^s$ with $s>1$ gives rise to one value $\beta(j)\beta(j+1)=exp(i\pi/3)$
and one value $\beta(j)\beta(j+1)=exp(-i\pi/3)$ -- totalling $1$.  All other values
of $\beta(j)\beta(j+1)$ are precisely $1$, so $k-\sum_j\beta(j)\beta(j+1)=c+e$
is equal to the number of blocks $(xy)^s$ or $(xy^2)^s$ for which $s>1$.

\medskip
Finally, since the total number of blocks is $2e+2$, it follows that $c+e\le 2e+2$,
or $c-2\le e$, as required.

\end{pf}

\subsection{The search algorithm}

Suppose given non-negative integers $c,e$ with $c+e>0$.  We wish to find all words
$W$ in $\Z_2\ast\Z_3=\<x,y|x^2=y^3=1\>$ whose trace polynomial is
$$\tau_W(\lambda)=\lambda(\lambda^2-2)^c(\lambda^4-3\lambda^2+1)^e.$$
By Lemma \ref{blocks} it suffices to consider words of the form
$$W=(xy)^{b(1)}(xy^2)^{b(2)}\cdots (xy)^{b(t-1)}(xy^2)^{b(t)}$$
with $t=2e+2$ such that $b(j)>1$ for precisely $c+e$ values of $j$ -- or equivalently
$b(j)=1$ for precisely $e+2-c$ values of $j$.  We encode $W$ as a list of
$t$ positive integers $[b(1),\dots,b(t)]$, of which $e+2-c$ are equal to $1$.
Since we are interested in classifying $W$ only up to cyclic permutation and inversion, it suffices also to consider lists only up to cyclic permutation and reversal.

In practice, we construct our list using GAP \cite{GAP} from two pieces of
input: a list $L$ of $c+e$  positive integers totalling $2c+2e-1$; and a subset
$C$ of $\{1,\dots,2e+2\}$ of size $e+2-c$.  The final list is obtained by adding $1$ to each element of $L$ and inserting a $1$ in the positions $C$.  This process
has two advantages: the sets of possible inputs $\mathcal{L},\mathcal{C}$ for
$L,C$ respectively are of more manageable size than the overall input set $\mathcal{L}\times\mathcal{C}$; and the resulting double programming loop facilitates
parallelisation of the process, making implementation possible in reasonable time.

Having constructed our list of integers, we use it to calculate $\tau_W(\lambda)$
and record only those lists that give the correct result.  As a final step,
we sift through the (now manageably small) list of recorded lists and delete
repetitions up to cyclic permutation and reversal.

For smaller values of $c$ and $e$, this algorithm can be carried out in
a reasonable time-frame with the standard version of GAP.  However, for
higher values it was essential to use the parallel version of GAP.
The longest search, with $(c,e)=(4,10)$, needed to test all possible lists
(up to cyclic permutation and reversal) of $22$ positive integers, such
that precisely $8$ entries are equal to $1$, and whose sum is $49$.  In practice,
$722030$ lists were tested by $32$ slave processors, each working for around $31240$ minutes (or just over $3$ weeks), to confirm that no word has trace polynomial
$\lambda(\lambda^2-2)^4(\lambda^4-3\lambda^2+1)^{10}$.

\newpage
\subsection{Search Results}

The table below shows all words (up to equivalence) with trace polynomial
of the form $\lambda(\lambda^2-2)^c(\lambda^4-3\lambda^2+1)^e$.

\medskip
\begin{center}
\begin{tabular}{|l||l|l|l|}
\hline
$n$ &$c$ & $e$ & $W_n$\\
\hline\hline
$1$ & $0$ & $0$ & $xy$\\
\hline
$2$ & $0$ & $1$ & $(xy)^2xy^2xyxy^2$\\
\hline
$3$ & $0$ & $2$ & $(xy)^3xy^2xy(xy^2)^2xyxy^2$\\
\hline
$4$ & $1$ & $0$ & $(xy)^2xy^2$\\
\hline
$5$ & $1$ & $1$ & $(xy)^3xy^2xy(xy^2)^2$\\
\hline
$6$ & $1$ & $2$ & $(xy)^3(xy^2)^3xyxy^2(xy)^2xy^2$\\
\hline
$7$ & $1$ & $2$ & $(xy)^4(xy^2)^2xyxy^2xy(xy^2)^2$\\
\hline
$8$ & $1$ & $3$ & $(xy)^4xy^2xy(xy^2)^3(xy)^2xy^2xy(xy^2)^2$\\
\hline
$9$ & $1$ & $4$ & $(xy)^4xy^2xy(xy^2)^3xy(xy^2)^2xyxy^2(xy)^2(xy^2)^3$\\
\hline
$10$ & $1$ & $4$ & $(xy)^4xy^2xy(xy^2)^3(xy)^3(xy^2)^2xy(xy^2)^2xyxy^2$\\
\hline
$11$ & $2$ & $0$ & $(xy)^3(xy^2)^2$\\
\hline
$12$ & $2$ & $1$ & $(xy)^3(xy^2)^3xy(xy^2)^2$\\
\hline
$13$ & $2$ & $2$ & $(xy)^4(xy^2)^3(xy)^2(xy^2)^2xyxy^2$\\
\hline
$14$ & $2$ & $3$ & $(xy)^4xy^2xy(xy^2)^2(xy)^3(xy^2)^3xy(xy^2)^2$\\
\hline
$15$ & $2$ & $3$ & $(xy)^4(xy^2)^3xyxy^2(xy)^2(xy^2)^3(xy)^2xy^2$\\
\hline
$16$ & $2$ & $3$ & $(xy)^4(xy^2)^3xy(xy^2)^3(xy)^2xy^2xy(xy^2)^2$\\
\hline
$17$ & $2$ & $4$ & $(xy)^4(xy^2)^2xyxy^2xy(xy^2)^2(xy)^3(xy^2)^4(xy)^2xy^2$\\
\hline
$18$ & $2$ & $4$ & $(xy)^4(xy^2)^3xy(xy^2)^2(xy)^3(xy^2)^3xyxy^2(xy)^2xy^2$\\
\hline
$19$ & $2$ & $4$ & $(xy)^4(xy^2)^4xyxy^2xy(xy^2)^2xy(xy^2)^2(xy)^3(xy^2)^2$\\
\hline
$20$ & $3$ & $1$ & $(xy)^4(xy^2)^3(xy)^2(xy^2)^2$\\
\hline
$21$ & $3$ & $2$ & $(xy)^4(xy^2)^3xy(xy^2)^2(xy)^3(xy^2)^2$\\
\hline
$22$ & $3$ & $2$ & $(xy)^4(xy^2)^3(xy)^2(xy^2)^3xy(xy^2)^2$\\
\hline
$23$ & $3$ & $4$ & $(xy)^4xy^2xy(xy^2)^4(xy)^2(xy^2)^3(xy)^2xy^2(xy)^2(xy^2)^3$\\
\hline
$24$ & $3$ & $6$ & $(xy)^4(xy^2)^4xyxy^2(xy)^2(xy^2)^3(xy)^3(xy^2)^2xy(xy^2)^3(xy)^3xy^2(xy)^2xy^2$\\
\hline
$25$ & $3$ & $6$ & $(xy)^5(xy^2)^3xy(xy^2)^2xyxy^2(xy)^3(xy^2)^4(xy)^2xy^2xy(xy^2)^2(xy)^3(xy^2)^2$\\
\hline
$26$ & $3$ & $8$ & $(xy)^5(xy^2)^4(xy)^3(xy^2)^2xyxy^2(xy)^2xy^2(xy)^3(xy^2)^3(xy)^2xy^2xy(xy^2)^4~\cdot$\\
& & & \hspace*{20mm} $
(xy)^2(xy^2)^2xyxy^2$\\
\hline
$27$ & $4$ & $4$ & $(xy)^4(xy^2)^3(xy)^2xy^2(xy)^3(xy^2)^2(xy)^3(xy^2)^4xy(xy^2)^2$\\
\hline
$28$ & $4$ & $4$ & $(xy)^4(xy^2)^3xy(xy^2)^2(xy)^3(xy^2)^4(xy)^2(xy^2)^3(xy)^2xy^2$\\
\hline
$29$ & $4$ & $4$ & $(xy)^4(xy^2)^4(xy)^2(xy^2)^3xy(xy^2)^2(xy)^3(xy^2)^3(xy)^2xy^2$\\
\hline
$30$ & $4$ & $5$ & $(xy)^4(xy^2)^3(xy)^2(xy^2)^2xy(xy^2)^4(xy)^4(xy^2)^2(xy)^3(xy^2)^2xyxy^2$\\
\hline
$31$ & $4$ & $6$ & $(xy)^4(xy^2)^2xyxy^2(xy)^2xy^2(xy)^3(xy^2)^4xy(xy^2)^3(xy)^4(xy^2)^3(xy)^2(xy^2)^2$\\
\hline
\end{tabular}

\vspace*{10mm}

Table 1\\
List of words with trace polynomial $\lambda(\lambda^2-2)^c(\lambda^4-3\lambda^2+1)^e$.
\end{center}

\section{Small Cancellation}\label{sc}

Several of the groups in Table 1 satisfy a small cancellation condition.
Recall that a word $U$ is a {\em piece} of a word $W$ (in the free product
$\Z_2\ast\Z_3$)
if there are distinct words
$V_1,V_2$ such that each of $U\cdot V_1$ and $U\cdot V_2$ is cyclically reduced as
written and
a cyclic permutation of $W^{\pm 1}$.
The table below expresses the words concerned (up to cyclic permutation)
as a product of $3$  non-pieces, each  of  even length at least $8$.  Following \cite[Corollary 2.3]{How} this is enough to
confirm the existence of non-abelian free subgroups.

\medskip
\begin{center}
\begin{tabular}{|l|l|}
\hline
$n$. &  $W_n$\\
\hline\hline
$9$ & $[(xy)^4] \cdot [xy^2xy(xy^2)^3xy] \cdot [(xy^2)^2xyxy^2(xy)^2(xy^2)^3]$\\
\hline
$10$ & $[(xy)^4] \cdot [xy^2xy(xy^2)^3xy] \cdot [(xy)^2(xy^2)^2xy(xy^2)^2xyxy^2]$\\
\hline
$14$ & $[(xy)^4] \cdot [xy^2xy(xy^2)^2(xy)^3xy^2] \cdot [(xy^2)^2xy(xy^2)^2] $\\
\hline
$15$ & $[(xy)^4] \cdot [(xy^2)^3xyxy^2xy]\cdot [xy(xy^2)^3(xy)^2xy^2]$\\
\hline
$17$ & $[(xy)^2(xy^2)^2xy] \cdot [xy^2xy(xy^2)^2(xy)^3xy^2] \cdot [(xy^2)^3(xy)^2xy^2(xy)^2]$\\
\hline
$18$ & $[(xy)^4] \cdot [(xy^2)^3xy(xy^2)^2(xy)^2] \cdot [xy(xy^2)^3xyxy^2(xy)^2xy^2]$\\
\hline
$19$ & $[(xy)^4(xy^2)^4xyxy^2] \cdot [xy(xy^2)^2xyxy^2] \cdot [xy^2(xy)^3(xy^2)^2]$\\
\hline
$22$ & $[(xy)^4] \cdot [(xy^2)^3(xy)^2(xy^2)^2] \cdot [xy^2xy(xy^2)^2]$\\
\hline
$23$ & $[xyxy^2xy(xy^2)^4(xy)^2xy^2] \cdot [(xy^2)^2(xy)^2xy^2xy] \cdot [xy(xy^2)^3(xy)^3]$\\
\hline
$24$ & $[(xy)^3(xy^2)^4xyxy^2xy] \cdot [xy(xy^2)^3(xy)^3(xy^2)^2xy(xy^2)^2] \cdot$\\
 & \hspace*{10mm} $ [xy^2(xy)^3xy^2(xy)^2xy^2xy]$\\
\hline
$25$ & $[(xy)^5] \cdot [(xy^2)^3xy(xy^2)^2] \cdot [xyxy^2(xy)^3(xy^2)^4(xy)^2xy^2xy(xy^2)^2(xy)^3(xy^2)^2]$\\
\hline
$26$ & $[(xy)^5] \cdot [(xy^2)^4(xy)^3] \cdot$\\
 & \hspace*{10mm} $ [(xy^2)^2xyxy^2(xy)^2xy^2(xy)^3(xy^2)^3(xy)^2xy^2xy(xy^2)^4
(xy)^2(xy^2)^2xyxy^2]$\\
\hline
$27$ & $[(xy)^4(xy^2)^3(xy)^2xy^2xy] \cdot [ (xy)^2(xy^2)^2xy] \cdot [(xy)^2(xy^2)^4xy(xy^2)^2]$\\
\hline
$28$ & $[(xy^2)^3xy(xy^2)^2(xy)^3(xy^2)^3] \cdot [xy^2(xy)^2(xy^2)^2] \cdot [xy^2(xy)^2xy^2(xy)^4]$\\
\hline
$30$ & $[(xy)^4(xy^2)^3] \cdot [(xy)^2(xy^2)^2xy(xy^2)^2] \cdot [(xy^2)^2(xy)^4(xy^2)^2(xy)^3(xy^2)^2xyxy^2$\\
\hline
$31$ & $[(xy)^4(xy^2)^2xy] \cdot [xy^2(xy)^2xy^2xy] \cdot$\\
 & \hspace*{10mm} $ [(xy)^2(xy^2)^4xy(xy^2)^3(xy)^4(xy^2)^3(xy)^2(xy^2)^2]$\\
\hline
\end{tabular}

\vspace*{10mm}

Table 2\\
Words from Table 1 satisfying the small cancellation condition.
\end{center}

\section{Almost small cancellation}\label{asc}

The group $G=G_{13}$ in our table does not satisfy a small cancellation condition, but comes sufficiently close for small-cancellation methods
to apply.

Thus while $G$ does not satisfy the small cancellation condition $C(6)$
(as a quotient of the free product $\Z_2*\Z_3$), it does satisfy $C(5)$
in a fairly strong form.  We will exploit this to obtain the existence of
a non-abelian free subgroup.

\begin{theorem}\label{free13}
Let $G=\<x,y|x^2=y^3=((xy)^4(xy^2)^3(xy)^2(xy^2)^2xyxy^2)^2=1\>$.
Let $N\gg 1$ be an integer.  Then $A:=(xy)^N$ and $B:=(xy^{-1})^N$ freely generate a free subgroup of $G$.
\end{theorem}

As the proof of this result is substantial and uses
technology which is different from the main thrust of the
paper, we have included it in Appendix C
below.

\section{Ad-hoc arguments}\label{ah}

By \S \ref{sc} and Theorem \ref{free13} we have $14$ exceptional words to consider, as follows.

\medskip
\begin{center}
\begin{tabular}{|l|l|}
\hline
$n$ & $W_n$\\
\hline\hline
$1$ & $xy$\\
\hline
$2$ & $(xy)^2xy^2xyxy^2$\\
\hline
$3$ & $(xy)^3xy^2xy(xy^2)^2xyxy^2$\\
\hline
$4$ & $(xy)^2xy^2$\\
\hline
$5$ & $(xy)^3xy^2xy(xy^2)^2$\\
\hline
$6$ & $(xy)^3(xy^2)^3xyxy^2(xy)^2xy^2$\\
\hline
$7$ & $(xy)^4(xy^2)^2xyxy^2xy(xy^2)^2$\\
\hline
$8$ & $(xy)^4xy^2xy(xy^2)^3(xy)^2xy^2xy(xy^2)^2$\\
\hline
$11$ & $(xy)^3(xy^2)^2$\\
\hline
$12$ & $(xy)^3(xy^2)^3xy(xy^2)^2$\\
\hline
$16$ & $(xy)^4(xy^2)^3xy(xy^2)^3(xy)^2xy^2xy(xy^2)^2$\\
\hline
$20$ & $(xy)^4(xy^2)^3(xy)^2(xy^2)^2$\\
\hline
$21$ & $(xy)^4(xy^2)^3xy(xy^2)^2(xy)^3(xy^2)^2$\\
\hline
$29$ & $(xy)^4(xy^2)^4(xy)^2(xy^2)^3xy(xy^2)^2(xy)^3(xy^2)^3(xy)^2xy^2$\\
\hline
\end{tabular}

\vspace*{10mm}

Table 3\\
Exceptional words.
\end{center}

\bigskip
Let us denote by $G_n$ the group
$$\<x,y|x^2=y^3=W_n(x,y)^2=1\>,$$
where $W_n$ is the word listed in Table 1.
The Rosenberger Conjecture is known for words with length parameter
$\ell\le 6$ \cite{Will}.  This covers groups $G_1,G_2,G_4$ and $G_{11}$.
Indeed, $G_1,G_2,G_4$ are finite of orders $6$, $720$ and $48$ respectively,
while $G_{11}$ is an extension of $\Z^3$ by $A_4$ -- see  Appendix B below
for a verification of this using GAP.

\medskip
Groups $G_3$, $G_7$ and $G_{29}$ can be shown to be large using GAP \cite{GAP}
-- see  Appendix B.
(Recall that a group $G$ is {\em large} if some subgroup $H<G$ of finite index
admits an epimorphism $H\twoheadrightarrow F_2$ onto the free group $F_2$ of
rank $2$.  Large groups clearly contain non-abelian free subgroups.)
Now the relators of $G_6$ are consequences of those of $G_3$, so it
follows that group $G_3$ is a homomorphic image of $G_6$.  Since
$G_3$ is large, $G_6$ is also large.

\begin{lemma}
Group $G_5$  contains a non-abelian free subgroup.
\end{lemma}

\begin{pf}
Group $G_5$ was shown to be infinite by L\'{e}vai, Rosenberger and Souvignier
\cite{LRS}.  Their proof constructs a representation $\rho:G'\to SL(3,\C)$
of the commutator subgroup $G'$, the image of which is generated by two matrices
$X,Y$ of order $3$. A more detailed analysis of the matrices in \cite{LRS}
shows that $XY$ (and hence also $YX$) has eigenvalues $+1,-1,-1$, with the
$(-1)$-eigenspace having dimension $1$. Specifically,
$$XY=\left(\begin{array}{rrl} -2& -1& t^{-1} \\ 1& 0& 0 \\ 0& 0& 1\end{array}\right),$$
where $t$ is an algebraic number with $t^6-3t^3+1=0$.
Let $V$ be the plane spanned by the eigenvectors of $XY$, namely
$v_{-1}:=(1,-1,0)^T$ and $v_{+1}:=(1,1,4t)^T$.  Then a calculation shows that $YX(v_{-1})\not\in V$ so $V$ is not invariant under $YX$.  If $V'$ is the plane
spanned by the eigenvectors of $YX$, then $V\ne V'$ so $L:=V\cap V'$ is a line.
Now $(XY)^2$ and $(YX)^2$ fix $V$ and $V'$ respectively (pointwise), and so they both fix $L$.
Under the induced action on the quotient plane $\C^3/L$, $(XY)^2$ and
$(YX)^2$ are parabolic with distinct fixed subspaces.  Hence they generate
a non-elementary subgroup of $PSL(2,\C)$.  This subgroup, and hence also $G$,
contains a non-abelian free subgroup.
\end{pf}

\begin{cor}
Each of the groups $G_8$, $G_{12}$, $G_{16}$ and $G_{21}$ contains a non-abelian
free subgroup.
\end{cor}

\begin{pf}
 It is an exercise to show that, for $j=8,12,16,21$,
the element of $G_5$ represented by $W_j^2$ is trivial.  Hence $G_5$ is a homomorphic
image of $G_j$.  Since $G_5$ contains a non-abelian free subgroup, so does
$G_j$.
\end{pf}

\section{Conclusion}

Putting together the results of \S\S \ref{search}, \ref{sc} and \ref{ah} we obtain

\begin{theorem}
Let $$G=\<x,y|x^2=y^3=W(x,y)^2=1\>$$ be a generalised triangle group, where the
exponent-sums of $x,y$ in $W$ are coprime to $2,3$ respectively.  Then either
$G$ contains a non-abelian free subgroup, or $G$ has a soluble subgroup of finite index, except possibly where $G$ is isomorphic to the group
$$G_{20}=\<x,y|x^2=y^3=((xy)^4(xy^2)^3(xy)^2(xy^2)^2)^2=1\>.$$
\end{theorem}

This, together with the observations on the even-length case in \S \ref{even} and
in Appendix A, complete the proof of Theorem \ref{main}.

\appendix{Even length relators}\label{app3}

In \cite{arxiv} the Rosenberger conjectured was verified for generalised triangle groups
of the form
$$G=\<x,y|x^2=y^2=W(x,y)^2=1\>$$
in which $x$ appears in $W$ with even exponent-sum, with precisely six exceptions.
The exceptional groups are defined by setting $W$ to be one of
\begin{enumerate}
\item $W_{7a}:=(xy)^4(xy^2)^3(xy)^2xy^2$;
\item $W_{7b}:=(xy)^3(xy^2)^2xy(xy^2)^2xyxy^2$;
\item $W_{9a}:=(xy)^5(xy^2)^3(xy)^2xy^2xy(xy^2)^2$;
\item $W_{12}:=(xy)^4(xy^2)^2xy(xy^2)^3(xy)^2xy^2xy(xy^2)^2$;
\item $W_{13a}:=(xy)^4(xy^2)^4xy(xy^2)^3(xy)^2xy^2xy(xy^2)^2$;
\item $W_{15a}:=(xy)^4(xy^2)^4xy(xy^2)^2xy(xy^2)^3(xy)^3(xy^2)^2xyxy^2$.
\end{enumerate}

Here we give the computational details of proofs that all of the above,
with the possible exception of $W_{7a}$, give rise to groups
containing non-abelian free subgroups.

Jack Button (private communication) informed us that he could prove, using his largeness-testing software \cite{Button}, that the groups with $W\in\{W_{9a},W_{12},W_{13a},W_{15a}\}$
are large -- and hence in particular contain non-abelian free subgroups.  Below is a log
of a GAP session verifying this fact.

\begin{verbatim}
gap> Epi:=function(P,ll)
> local i,gg,Q;
> Q:=ShallowCopy(P);
> gg:=GeneratorsOfPresentation(Q);
> for i in ll do if i<1+Length(gg) then AddRelator(Q,gg[i]); fi; od;
> TzGoGo(Q);
> TzPrint(Q);
> return(Q);
> end;
function( P, ll ) ... end
gap> F:=FreeGroup(2);; x:=F.1;; y:=F.2;; U:=x*y;; V:=x*y^2;;
gap> W9:=U^5*V^3*U^2*V*U*V^2;;
gap> W12:=U^4*V^2*U*V^3*U^2*V*U*V^2;;
gap> W13:=U^4*V^4*U*V^3*U^2*V*U*V^2;;
gap> W15:=U^4*V^4*U*V^2*U*V^3*U^3*V^2*U*V;;
gap> G9:=F/[x^2,y^3,W9^2];;
gap> G12:=F/[x^2,y^3,W12^2];;
gap> G13:=F/[x^2,y^3,W13^2];;
gap> G15:=F/[x^2,y^3,W15^2];;
gap> x:=G9.1;; y:=G9.2;;
gap> H9:=Subgroup(G9,[x,(y*x)^2*(y*x^-1)^2*y,y*x*y^-1*x*(y*x^-1)^2*y^-1,
 (y*x)^2*y^-1*x*y^-1*x^-1*y*x^-1*y^-1*x^-1*y ]);;
gap> x:=G12.1;; y:=G12.2;;
gap> H12:=Subgroup(G12,[x,(y*x)^2*(y*x^-1)^2*y,y*x*y^-1*x*(y*x^-1)^2*y^-1,
 (y*x)^2*y^-1*x*y^-1*x^-1*y*x^-1*y^-1*x^-1*y ]);;
gap> P9:=PresentationSubgroup(G9,H9);
<presentation with 6 gens and 14 rels of total length 134>
gap> P12:=PresentationSubgroup(G12,H12);
<presentation with 6 gens and 14 rels of total length 154>
gap> Index(G9,H9);
20
gap> Index(G12,H12);
20
gap> Epi(P9,[4,6]);
#I  there are 2 generators and 1 relator of total length 2
#I  generators: [ _x1, _x2 ]
#I  relators:
#I  1.  2  [ 1, 1 ]
<presentation with 2 gens and 1 rels of total length 2>
gap> Epi(P12,[4,6]);
#I  there are 2 generators and 1 relator of total length 2
#I  generators: [ _x1, _x2 ]
#I  relators:
#I  1.  2  [ 1, 1 ]
<presentation with 2 gens and 1 rels of total length 2>
gap> H13:=Subgroup(G13,[(G13.1*G13.2)^5]);;
gap> P13:=PresentationNormalClosure(G13,H13);
<presentation with 11 gens and 30 rels of total length 396>
> gap> Epi(P13,[1..4]);
gap> Epi(P13,[1..4]);
#I  there are 2 generators and 0 relators of total length 0
#I  generators: [ _x5, _x6 ]
#I  there are no relators
<presentation with 2 gens and 0 rels of total length 0>
gap> H15:=Subgroup(G15,[(G15.1*G15.2)^5]);;
gap> P15:=PresentationNormalClosure(G15,H15);
<presentation with 11 gens and 30 rels of total length 484>
gap> Epi(P15,[1..4]);
#I  there are 2 generators and 1 relator of total length 8
#I  generators: [ _x5, _x8 ]
#I  relators:
#I  1.  8  [ -2, 1, 2, 2, 1, -2, -1, -1 ]
<presentation with 2 gens and 1 rels of total length 8>
gap> #
gap> # Each group G9, G12, G13, G15 has a finite index subgroup that admits
gap> # an epimorphism onto a large group.
gap> # Therefore each is itself large.
gap> #
\end{verbatim}

In the case of the group defined by $W=W_{7b}$, while we are unable to prove largeness, we
can verify the existence of non-abelian free subgroups (and hence the Rosenberger conjecture)
by constructing an irreducible essential representation $\sigma$ to $PSL(2,\C)$
from the commutator subgroup $[G,G]$ of  $G$ (which has index $6$). This is
outlined in the proof of Theorem \ref{G7b}. The computational details are
as follows.  First we use GAP \cite{GAP} to obtain the presentation of
$[G,G]$ stated in the proof of Theorem \ref{G7b}:

\begin{verbatim}
gap> #
gap> #  Calculation of presentation of derived subgroup [G,G] of 
gap> #
gap> #    G = < x, y | x^2 = y^3 = W^2 = 1 >, where
gap> #
gap> #  W = (xy)^3 (xy^2)^2 xy (xy^2)^2 xy xy^2
gap> #
gap> F:=FreeGroup(["x","y"]);; x:=F.1;; y:=F.2;;
gap> W:=(x*y)^3*(x*y^2)^2*x*y*(x*y^2)^2*x*y*x*y^2;;
gap> G:=F/[x^2,y^3,W^2];;
gap> D:=DerivedSubgroup(G);;
gap> P:=PresentationSubgroup(G,D);;
gap> TzPrint(P);
#I  generators: [ _x1, _x2 ]
#I  relators:
#I  1.  12  [ 2, 2, 1, -2, 1, 1, 2, 2, 1, -2, 1, 1 ]
#I  2.  12  [ 2, 2, 1, 1, -2, 1, 2, 2, 1, 1, -2, 1 ]
#I  3.  14  [ 2, 2, -1, -1, -1, 2, -1, 2, 2, -1, -1, -1, 2, -1 ]
#I  4.  14  [ 2, 2, -1, 2, -1, -1, -1, 2, 2, -1, 2, -1, -1, -1 ]
#I  5.  14  [ 2, 2, -1, 2, -1, 2, 1, 2, 2, -1, 2, -1, 2, 1 ]
#I  6.  14  [ 2, 2, 1, 2, -1, 2, -1, 2, 2, 1, 2, -1, 2, -1 ]
gap> #
\end{verbatim}

This confirms that $[G,G]$ is presented on two generators $u,v$
by six defining
relators of the form $W_j^2$, $j=1,\dots,6$.  To check that the two
matrices \\
\centerline{$M=\left(\begin{array}{cc} 0 & -1\\ 1 & a\end{array}\right),~~~
N=\left(\begin{array}{cc} -z & 1 + bz + z^2\\ -1 & z+b\end{array}\right)$}

\noindent
in the proof of Theorem \ref{G7b} do indeed give a representation $\sigma$ to
$PSL(2,\C)$, we need to ensure that $W_j(M,N)$ has trace $0$ in $\Q[z]$
for each $j$.  In the MAPLE session illustrated below we verify
that $W_1(M,N)$ and $W_2(M,N)$ have the equal trace (as a polynomial in $a,b,z$) which we define to be $f(a,b,z)$.  Similarly $W_3(M,N)$ and $W_4(M,N)$ have equal trace $g(a,b,z)$, while $W_5(M,N)$ and $W_6(M,N)$ 
have equal trace $h(a,b,z)$.

\begin{center}
\includegraphics[scale=0.76]{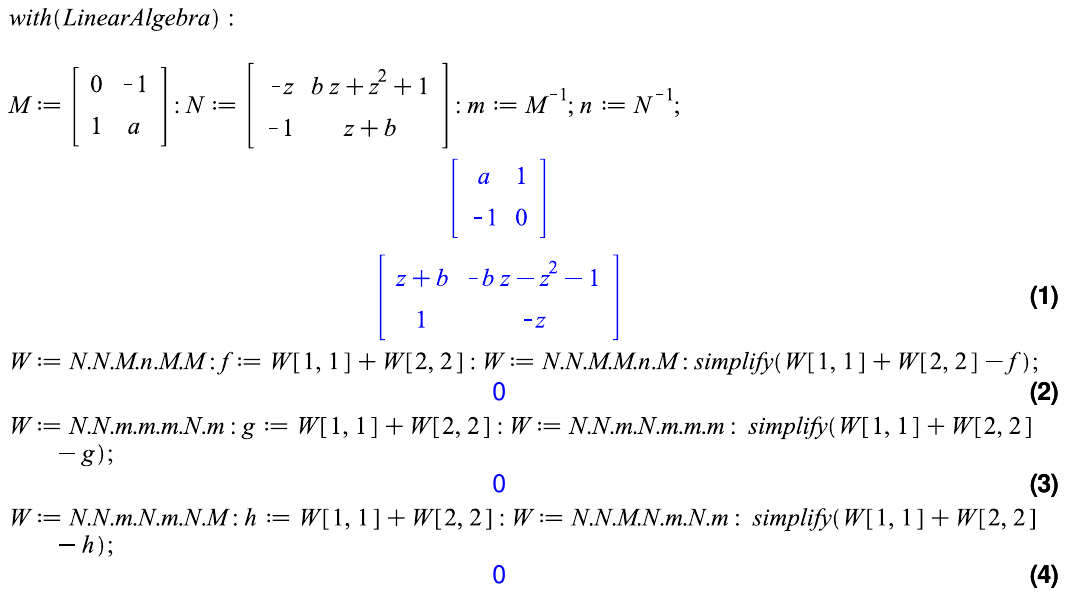}
\end{center}

We next set\\
\centerline{$p(X):=X^{12}-3X^{10}-9X^8+42X^6-48X^4+15X^2+1,$}

\noindent
and compute a Gr\"{o}bner basis $B$ for the ideal $J$ of the polynomial
ring $R:=\Q[a,b,z]$ generated by $f,g,h,p(a)$.
The first term in $B$ is a polynomial $q(z)$ of degree $48$, which we 
confirm to be irreducible.  The second term is a polynomial in $b,z$ which is linear in $b$, and the third term is linear in $a$.  Hence the quotient
ring $K:=R/J$ is an algebraic extension of $\Q$ of degree $48$, generated by $z$.  This establishes that $\sigma:u\mapsto M,~~v\mapsto N$ defines a representation from $[G,G]$ to $PSL(2,K)$.

Using a second Groebner basis calculation, we show that
the ideal of $R$ generated by $B$ together with
$d:=a^2+b^2+c^2+abc-4$ (where $c$ is the trace of $MN$) is the whole of
$R$.  This confirms that $\sigma$ is an irreducible representation.

Finally, we check that $p$ is irreducible, and that both $a,b\in K$ are roots
of $p$.  It follows that none of the eigenvalues
of $M,N$ can be an $n$'th root of unity for any $n$ with Euler totient $\phi(n)<2\cdot\mathrm{degree}(p)=24$, and hence that each of $M,N$ has order greater than $12$ in $PSL(2,K)$.
In particular, $\sigma([G,G])=\< M,N\>$ cannot be dihedral or isomorphic to one of
the finite subgroups $A_4,S_4,A_5$ of $PSL(2,\C)$.

This is sufficient to confirm that $\sigma([G,G])$ is non-elementary and so contains a non-abelian free subgroup.

\begin{center}
\includegraphics[scale=0.76]{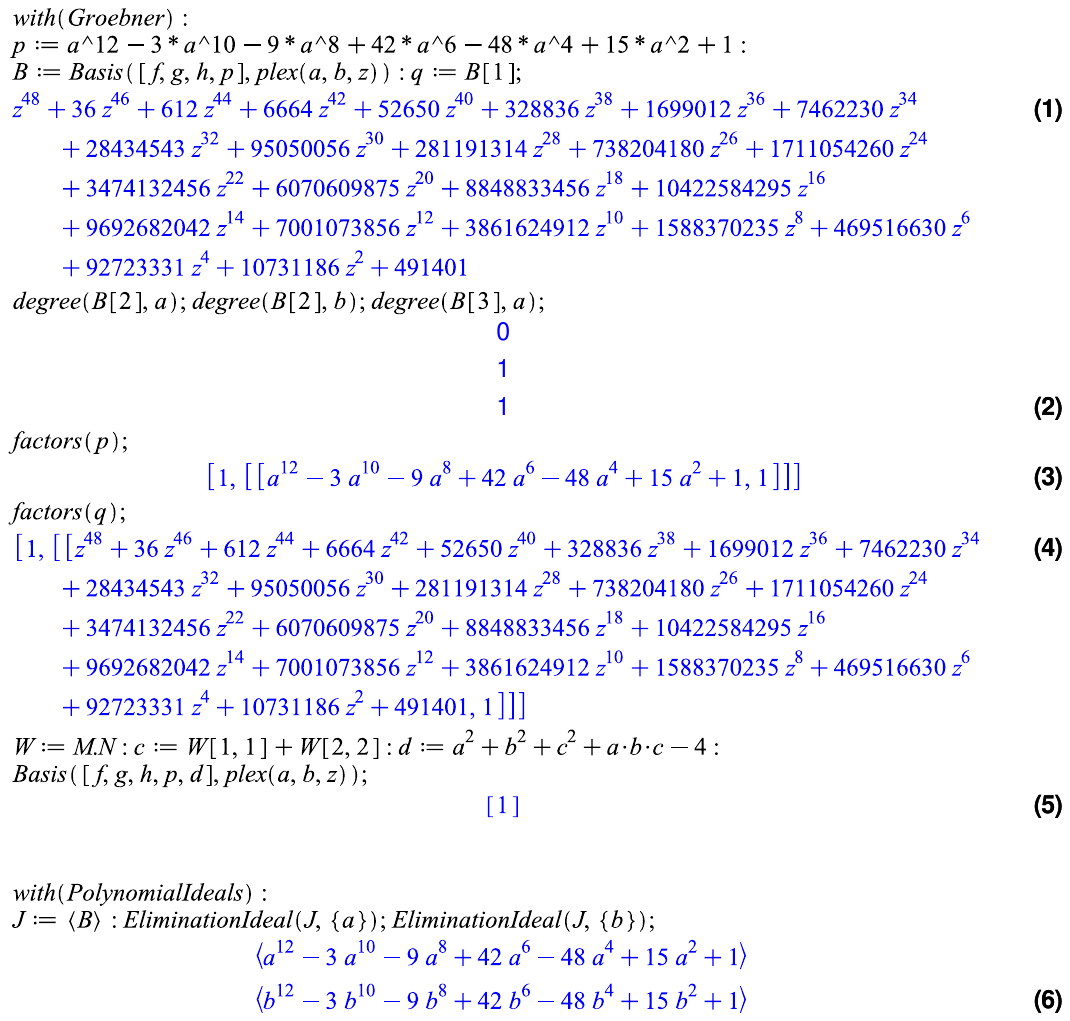}
\end{center}

\appendix{GAP calculations in the odd length case}

\subsection*{Verification that $G_{11}$ is abelian-by-finite}

\begin{verbatim}
gap> F:=FreeGroup(["x","y"]);; x:=F.1;; y:=F.2;;
gap> W:=(x*y)^2*(x*y^2)^3;;
gap> G:=F/[x^2,y^3,W^2];;
gap> H:=Subgroup(G,[(G.1*G.2)^4]);;
gap> P:=PresentationNormalClosure(G,H);
<presentation with 4 gens and 11 rels of total length 48>
gap> SimplifyPresentation(P);
#I  there are 3 generators and 3 relators of total length 14
#I  there are 3 generators and 3 relators of total length 12
gap> TzPrint(P);
#I  generators: [ _x1, _x2, _x3 ]
#I  relators:
#I  1.  4  [ 1, 2, -1, -2 ]
#I  2.  4  [ 1, 3, -1, -3 ]
#I  3.  4  [ 2, 3, -2, -3 ]
\end{verbatim}

The kernel of a representation $G_{11}\to S_4$ is free abelian of rank $3$.

\subsection*{Verification that $G_3$ is large}

\begin{verbatim}
gap> F:=FreeGroup(["x","y"]);; x:=F.1;; y:=F.2;;
gap> W:=(x*y)^3*x*y^2*x*y*(x*y^2)^2*x*y*x*y^2;;
gap> G:=F/[x^2,y^3,W^2];;
gap> H:=Subgroup(G,[(G.1*G.2)^5]);;
gap> P:=PresentationNormalClosure(G,H);
<presentation with 10 gens and 29 rels of total length 196>
gap> SimplifyPresentation(P);
#I  there are 7 generators and 24 relators of total length 236
#I  there are 7 generators and 24 relators of total length 214
gap> gg:=GeneratorsOfPresentation(P);;
gap> AddRelator(P,gg[7]);
gap> AddRelator(P,gg[4]*gg[6]^-1);
gap> for j in [1,2,3,5] do AddRelator(P,gg[1]*gg[j]); od;
gap> SimplifyPresentation(P);
#I  there are 2 generators and 1 relator of total length 2
gap> TzPrint(P);
#I  generators: [ _x1, _x5 ]
#I  relators:
#I  1.  2  [ 1, 1 ]
\end{verbatim}

The kernel of a representation $G_3\to A_5$ has the large group $\Z_2\ast\Z$ as
a homomorphic image.

\subsection*{Verification that $G_7$ is large}

\begin{verbatim}
gap> F:=FreeGroup(["a","b"]);; a:=F.1;; b:=F.2;;
gap> W:=(a*b)^4*(a*b^2)^2*a*b*a*b^2*a*b*(a*b^2)^2;;
gap> G:=F/[a^2,b^3,W^2];;
gap> H:=Subgroup(G,[(G.1*G.2)^5]);;
gap> Q:=PresentationNormalClosure(G,H);
<presentation with 11 gens and 30 rels of total length 242>
gap> SimplifyPresentation(Q);
#I  there are 7 generators and 26 relators of total length 438
#I  there are 7 generators and 26 relators of total length 432
gap> gg:=GeneratorsOfPresentation(Q);
[ _x1, _x3, _x4, _x6, _x8, _x10, _x11 ]
gap> AddRelator(Q,gg[3]*gg[6]);
gap> for j in [1,2,5,7] do AddRelator(Q,gg[j]); od;
gap> SimplifyPresentation(Q);
gap> K:=FpGroupPresentation(Q);;
gap> N:=Subgroup(K,[K.1,K.2^2]);;
gap> P:=PresentationNormalClosure(K,N);
<presentation with 3 gens and 2 rels of total length 8>
gap> hh:=GeneratorsOfPresentation(P);
[ _x1, _x2, _x3 ]
gap> AddRelator(P,hh[2]);
gap> SimplifyPresentation(P);
#I  there are 2 generators and 0 relators of total length 0
gap> # The free group presented by P is a homomorphic image
gap> # of a subgroup of index 120 in G7.  Hence G7 is large.
gap> # QED.
\end{verbatim}

\subsection*{Verification that $G_{29}$ is large}

\begin{verbatim}
gap> F:=FreeGroup(["x","y"]);; x:=F.1;; y:=F.2;;
gap> U:=x*y;; V:=x*y^2;;
gap> W:=U^4*V^4*U^2*V^3*U*V^2*U^3*V^3*U^2*V;;
gap> G:=F/[x^2,y^3,W^2];;
gap> H:=Subgroup(G,[(G.1*G.2)^5]);;
gap> P:=PresentationNormalClosure(G,H);
<presentation with 11 gens and 30 rels of total length 550>
gap> gg:=GeneratorsOfPresentation(P);;
gap> for j in [2,4,6,8,9] do AddRelator(P,gg[j]); od;
gap> SimplifyPresentation(P);
#I  there are 2 generators and 0 relators of total length 0
\end{verbatim}

The kernel of a representation $G_{29}\to A_5$ has the free group of rank $2$ as
a homomorphic image.

\subsection*{Details of proof of Lemma 4}

\begin{verbatim}
gap> #
gap> # Proof that <x,y|x^3 = y^3 = xyxy^2xy^2x^2yxyx^2yx^2y^2 = 1>
gap> # contains a nonabelian free subgroup, using the
gap> # Levai-Rosenberger-Souvignier representation
gap> #
gap> R:=UnivariatePolynomialRing(Rationals,"t");;
gap> t:=IndeterminatesOfPolynomialRing(R)[1];;
gap> a:=-3*t^4+8*t;;
gap> b:=-4*t^4+11*t;;
gap> c:=2*t^3-6;;
gap> d:=-5*t^5+14*t^2;;
gap> e:=-7*t^5+19*t^2;;
gap> f:=t^6-3*t^3+1;;
gap> x:=[[a,b,c],[0,0,1],[d,e,-a]];;
gap> y:=[[d,e,-a],[3*(b*t-d),-d,-c*t],[1,0,0]];; 
gap> # 
gap> # Representation modulo f - ie
gap> #      G' -> GL(V)=GL(3,K) with |K,Q|=6.
gap> # 
gap> # Next check the relators
gap> #
gap> x^3 mod f;
[ [ 1, 0, 0 ], [ 0, 1, 0 ], [ 0, 0, 1 ] ]
gap> y^3 mod f;
[ [ 1, 0, 0 ], [ 0, 1, 0 ], [ 0, 0, 1 ] ]
gap> x*y*x*y^2*x*y^2*x^2*y*x*y*x^2*y*x^2*y^2 mod f;
[ [ 1, 0, 0 ], [ 0, 1, 0 ], [ 0, 0, 1 ] ]
gap> Id := last;    # identity matrix
[ [ 1, 0, 0 ], [ 0, 1, 0 ], [ 0, 0, 1 ] ]
gap> m:=x*y mod f; 
[ [ -2, -1, -t^5+3*t^2 ], [ 1, 0, 0 ], [ 0, 0, 1 ] ]
gap> # Clearly det(m)=1 and tr(m)= -1
gap> # in fact m has eigenvalues +1, -1, -1:
gap> Rank(m-Id);
2
gap> Rank(m+Id);
2
gap> # and the -1 eigenspace is 1-dimensional
gap> # hence m^2 is nontrivial parabolic 
gap> # (this is the L-R-S proof that the group is infinite)
gap> # Next find eigenvectors
gap> ev1:=[[1],[1],[4*t]];;
gap> ev2:=[[1],[-1],[0]];;
gap> m*ev1 mod f;
[ [ 1 ], [ 1 ], [ 4*t ] ]
gap> m*ev2 mod f;
[ [ -1 ], [ 1 ], [ 0 ] ]
gap> # Defining equation for the span P of the eigenvectors
gap> q:=[2*t,2*t,-1];;
gap> q*ev1;
[ 0 ]
gap> q*ev2;
[ 0 ]
gap> # now construct a conjugate of m
gap> n:=y*x mod f;;
gap> # n also has eigenvalues +1.-1,-1, with 1-dimensional
gap> # (-1)-eigenspace we check that the plane P' spanned by
gap> # its eigenvectors does not coincide with the span P of the
gap> # eigenvectors of m. To do this, show that P is not
gap> # invariant under n
gap> q*n*ev1 mod f;
[ t^4+t ]
gap> q*n*ev2 mod f;
[ t^4-t ]
gap> # Thus P and P' are distinct, so intersect in a line L.
gap> # m^2 and n^2 fix L.  They act on the quotient V/L as parabolics
gap> # with distinct fixed spaces, so they generate a
gap> # non-elementary subgroup of PSL(2,C).   QED.
\end{verbatim}

\appendix{Pictures}\label{picts}

In this appendix we recall the theory of pictures over a free product of groups, and use it to prove Theorem \ref{free13}.

Suppose that $\G_1,\G_2$ are groups, and $U\in\G_1\ast\G_2$ is a cyclically reduced
word of length at least $2$. (Here and throughout this appendix, {\em length}
means length in the free product sense.)
A word $V\in\G_1\ast\G_2$ is called a {\em piece}
if there are words $V',V''$ with $V'\ne V''$, such that each of 
$V\cdot V'$, $V\cdot V''$ is cyclically reduced as written, and each is
equal to a cyclic conjugate of $U$ or of $U^{-1}$.  A cyclic subword
of $U$ is a {\em non-piece} if it is not a piece.

By a {\em one-relator product} $(\G_1\ast\G_2)/U$ of groups $\G_1,\G_2$
we mean the quotient of their free product $\G_1\ast\G_2$ by the normal
closure of a cyclically reduced word $U$ of positive length.
Recall \cite{H1} that a {\em picture} over the one-relator product
$G=(\G_1\ast\G_2)/U$ is a graph $\mathcal{P}$ on a surface $\Sigma$ (which
for our purposes will always be a disc) whose corners are labelled
by elements of $\G_1\cup\G_2$, such that
\begin{enumerate}
 \item the label around any vertex, read in clockwise order, spells out a cyclic
permutation of $U$ or $U^{-1}$;
\item the labels in any region of $\Sigma\smallsetminus\mathcal{P}$
either all belong to $\G_1$ or all belong to $\G_2$;
\item if a region has $k$ boundary components labelled by words $W_1,\dots,W_k\in\G_i$
(read in anti-clockwise order; with $i=1,2$), then the quadratic equation
$$\prod_{j=1}^k X_jW_jX_j^{-1}=1$$
is solvable for $X_1,\dots,X_k$ in $\G_i$.
(In particular, if $k=1$ then $W_1=1$ in $\G_i$).
\end{enumerate}

Note that edges of $\mathcal{P}$ may join vertices to vertices,
or vertices to the boundary $\partial\Sigma$,
or $\partial\Sigma$ to itself, or may be simple closed curves disjoint from the
rest of $\mathcal{P}$ and from $\partial\Sigma$.

Pictures may be modified using {\em bridge moves} (Figure \ref{bridgemove}),
defined as follows.  Let $\gamma$ denote an
arc in the surface $\Sigma$, that meets the picture $\mathcal{P}$ only in
its endpoints, which are interior points of arcs of $\Gamma$.  A bridge
move is the result of altering $\Gamma$ by surgery along $\gamma$.
It is allowed provided that the resulting picture satisfies the
above rules concerning labels within a region.  The only
example of this that we will use in practice is that $\gamma$ divides a simply-connected
region (say a $\G_1$-region) into two parts.  The requirement for a bridge
move is then that, in the resulting subdivision of the region-label into
two subwords, each of the two subwords represent the identity in $\G_1$.

\begin{center}
\begin{figure}
\includegraphics[scale=0.3]{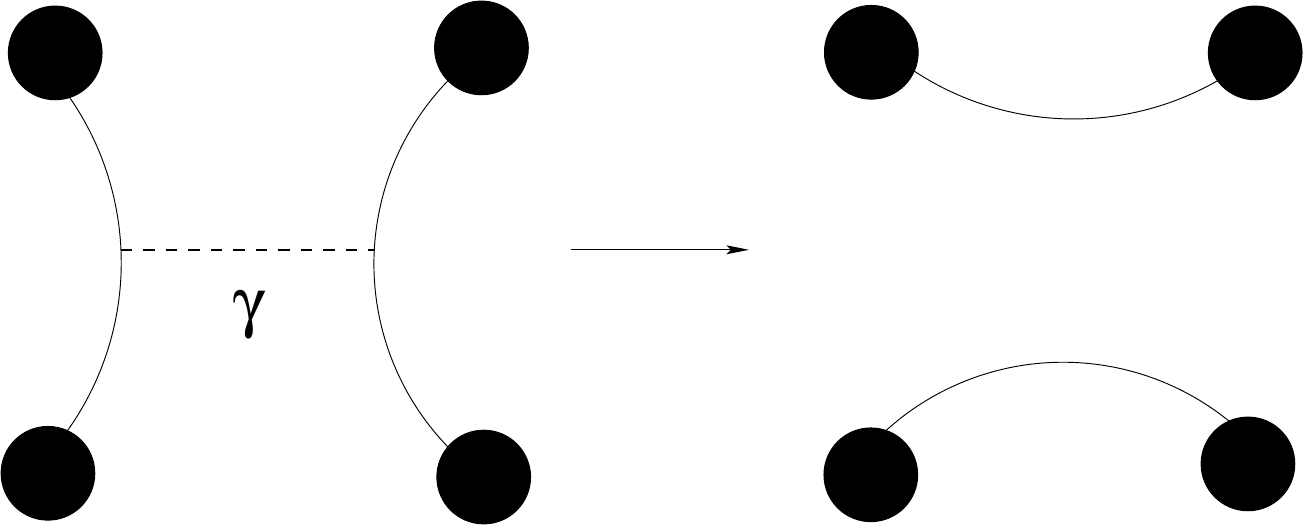}
\caption{A bridge move}
\label{bridgemove}
\end{figure}
\end{center}

The {\em boundary label} of $\mathcal{P}$ is the product of the labels around $\partial\Sigma$.
By a version of van Kampen's Lemma, there is a picture on the disc
with boundary label
$W\in\G_1\ast\G_2$ if and only if $W$ belongs to the normal closure of $U$.

A picture is {\em minimal} if it has the fewest possible vertices among
all pictures with the same (or conjugate) boundary labels.  In particular
every minimal picture is
{\em reduced}: no edge $e$ joins two distinct vertices in such a way that the labels
of these two vertices that start and finish at the endpoints of $e$
are mutually inverse.

In a reduced picture, any collection of parallel edges between two vertices
(or from one vertex to itself)
corresponds to a collection of consecutive $2$-gonal regions, and the labels within these
$2$-gonal regions spell out a piece (Figure \ref{piece}).

Since $U$ is cyclically reduced, no corner of an interior vertex is contained in a $1$-gonal region.

Recall the statement of Theorem \ref{free13}:

\medskip\noindent{\bf Theorem \ref{free13}}
{\em Let $N\gg 1$ be an integer.  Then $A:=(xy)^N$ and $B:=(xy^{-1})^N$ freely generate a free subgroup of $G$.}

\medskip
To prove the theorem, we suppose the conclusion false and derive a
contradiction.  There must be a minimal disc-picture
$\mathcal{P}_1$ with boundary label
a cyclically reduced word in $\{A,B\}$ (rewritten as a cyclically reduced word
$Z_1$ in $\Z_2*\Z_3$). 

\medskip
Note that there is very limited cancellation possible
when re-writing a cyclically reduced word in $\{A,B\}$ as
a cyclically reduced word in $\Z_2\ast\Z_3$.  In particular
the word $Z_1$  has no cyclic subword of the form
$(y(xy^2)^nxy)^{\pm 1}$ for $2\le n\le N-3$.

\medskip
To simplify our analysis of pictures, we note that we
can apply bridge-moves to make every $\Z_2$-region a $2$-gon.
We then suppress the $\Z_2$-regions, replacing each by a single arc.

\medskip
We will also make the following assumption: among all
minimal pictures with boundary label $Z(A,B)$ (and with
$\Z_2$-regions suppressed as above, we have chosen $\mathcal{P}$ to maximise the number of triangular $\Z_3$-regions.

\medskip
We will further simplify our picture by amalgamating
each maximal collection of parallel arcs into a single {\em edge} $e$.  We say that the piece $P$ defined by the collection of parallel arcs represented by $e$ at a vertex $v$ is {\em carried} by $e$ at $v$.  The corresponding piece at the other end of $e$ will be called the {\em match} of $P$.

Up to cyclic permutation in $\Z_2*\Z_3$ we can write
$$W=W_{13}=xy^{\a(0)}x\cdots xy^{\a(C)},$$
where the $\a{i}$ are indexed hexadecimally $0,1,\dots,C$ such that
$\a(0)=\a(2)=\a(3)=\a(4)=\a(5)=\a(9)=\a(A)=+1$ and
$\a(1)=\a(6)=\a(7)=\a(8)=\a(B)=\a(C)=-1$.

Without further comment we will interpret the indices modulo $13$ ($D$
in hexadecimal), and use the interval notation $[i..j]$ as shorthand for
the cyclic subword $xy^{\a(i)}x\cdots xy^{\a(j)}$ of $W$. 
We will use the bar notation $\overline{[i..j]}$ for the inverse of
$[i..j]$.  Then the product
$$[0..4]\cdot[5..9]\cdot[A..1]\cdot[2..5]\cdot[6..A]$$
of five non-pieces represents a proper subword of $W^2$.

It is also worth remarking that $[7..B]$ is a non-piece.

\begin{lemma}\label{ind5}
Let $v$ be a vertex in a reduced picture over $G$ that is not joined by arcs to the boundary. Then $v$ has index at least $5$.
If $v$ has index $5$ then the corners (in cyclic order) have labels from the
sets $\{2,3,4\}$, $\{7,8,9\}$, $\{0,1\}$, $\{5\}$, $\{A\}$.
\end{lemma}

\begin{pf}
This follows from the above remarks about non-pieces.  Indeed from the
decomposition of (a subword of) $W^2$ into $5$ non-pieces, each corner label
must belong to the intersection of two non-adjacent non-pieces of the
decomposition, namely $[0..4]\cap[A..1]=\{0,1\}$, $[5..9]\cap[2..5]=\{5\}$,
$[A..1]\cap[6..A]=\{A\}$, $[2..5]\cap[0..4]=\{2,3,4\}$ and
$[6..A]\cap[5..9]=\{6,7,8,9\}$.  The fact that $[7..B]$ is also a non-piece
allows us to discard $6$ as a possible corner label, giving the result.
\end{pf}

We may apply the above result to our minimal picture $\mathcal{P}$.
 But we may also
apply it to other reduced pictures; we exploit this 
in the proof of Lemma \ref{last} below
to obtain information
about boundary vertices of $\mathcal{P}$. 

\medskip
Here is one more simple result which will be useful later.

\begin{lemma}\label{no239}
If $v$ is an interior vertex of degree $5$, then no
triangular corner of $v$ can be corner $2,3$ or $9$.
\end{lemma}

\begin{pf}
By Lemma \ref{ind5}, we know that if $2$ or $3$ is a corner at $v$ in a region $R$, then $A$ is an adjacent corner at $v$.  The edge between these corners carries a piece $[B..1]$ or $[B..2]$.
In each case the unique match $\overline{[0..3]}$ or $\overline{[C..3]}$ extends to a match $\overline{[B..3]}$ for $[B..3]$.  So the other end $u$ of $e$ has the same orientation as $v$, and corner $A$ in $R$.  In particular
$R$ has two consecutive corners with opposite labels.  This is not possible if $R$ is triangular.

A similar argument applies if $v$ has a corner $9$ in $R$. There is an adjacent corner $0$ or $1$ at $v$, the edge between these corners carries $[A..C]$ or $[A..0]$ at $v$.
Each of the possible matches, namely $[5..7]$,
$\ol{[4..6]}$ and $\ol{[9..B]}$ for $[A..C]$ or  $\ol{[8..B]}$ for $[A..0]$, extends to a match for $[9..C]$, or $[9..0]$ so
 the labels of the corners of $R$ at $u,v$ cancel, and hence $R$ cannot be a triangular region.
\end{pf}

\subappendix{Angles and Curvature}

Assign angles $\theta(c)$ to corners $c$ of $\mathcal{P}$ as follows.

\begin{enumerate}
\item A corner of a boundary region between a boundary edge and a non-boundary edge has angle $\pi/2$.
\item A corner of a boundary region between two non-boundary edges has angle $\pi$.
\item A corner between a boundary edge and $\partial D^2$ has angle $\pi/2$.
\item A corner of a non-simply connected region has angle $\pi$.
\item If an interior vertex of degree $5$ has precisely one
non-triangular corner, then each such corner has angle $2\pi/3$.
\item If an interior vertex of degree $5$ has two or more
non-triangular corners, then each such corner has angle $\pi/2$.
\item All other corners have angle $\pi/3$.
\end{enumerate}

The angle assignments lead to a measure of curvature for vertices
and regions of $\mathcal{P}$, as follows.

\begin{enumerate}
\item The curvature $\kappa(v)$ of a vertex $v$ is $2\pi$ less the sum of the
angles of corners at $v$.
\item The curvature $\kappa(R)$ of an interior region $R$ is
$2\pi\chi(R)+\sum_c (\theta(c)-\pi)$, where $\chi$ denotes Euler characteristic, and the sum is over all corners $c$ in $R$.
\item The curvature $\kappa(R)$ of a boundary region $R$ is
$\pi\chi(R)+\sum_c (\theta(c)-\pi)$, where $\chi$ denotes Euler characteristic, and the sum is over all corners $c$ in $R$.
\end{enumerate}

It is an immediate consequence of Euler's formula that the total
curvature, summed over all the vertices and regions of $\mathcal{P}$,
is $+2\pi$.

\begin{theorem}\label{posvertex}
Let $v$ be an interior vertex of positive $\kappa$-curvature: $\kappa(v)>0$.
Then the sequence of non-$2$-gonal
corner labels (in cyclic order) around $v$ is
 $5,A,4,7,1$.
Moreover the pieces $[2..4]$, $[6..9]$ at $v$ are matched to
$[3..5]$ and $\overline{[1..4]}$ respectively at the corresponding neighbouring vertices.
\end{theorem}

\begin{pf}
By Lemmas \ref{ind5} and \ref{no239} the vertex $v$ must have index $5$ and has three consecutive corners $5,A,4$,
with the fourth and fifth corners being respectively $7$ or $8$ and $0$ or $1$.  Since $\kappa(v)>0$
it follows that all five of the corners incident at $v$ belong to
triangular regions.  In a triangular region the corner labels are all
equal (to $y$ or to $y^{-1}$).  This allows us to eliminate most of the
possible combinations of corner labels given in Lemma \ref{ind5}.

If the fourth corner of $v$ is $8$, then the edge between the
third and fourth corners of $v$ carries $[5..7]$.
Each possible match $[A..C]$, $\ol{[4..6]}$, $\ol{[9..B]}$ extends
to a match for $[4..7]$, so the third corner cannot be triangular. This contradiction shows that the fourth corner of $v$ must be $7$.

\medskip
A similar argument applies if the fifth corner is $0$.
The edge between the fourth and fifth corners carries
$[8..C]$.  The unique match $\ol{[9..0]}$ for this extends
to a match $\ol{[8..0]}$ for $[8..0]$, so the fifth
corner is not triangular.

This shows that the corners of $v$ are $5,A,4,7,1$, as claimed.

\medskip
We must now show that the matches for the pieces $[2..4]$ and
$[6..9]$ are as claimed.

Consider the piece $[2..4]$, represented by the group of arcs between the
fifth and first corners. The two possible matches for it are $[3..5]$ and
$\ol{[6..8]}$.  But the latter extends to a match $\ol{[6..9]}$ for $[1..4]$.
Since the fifth corner is triangular, the match must be $[3..5]$.

The piece $[6..9]$, represented by the group of arcs between
the first and second corners, has only one possible match, namely $\ol{[1..4]}$.

\end{pf}

 We next show that positive curvature does not arise in regions.
 
\begin{prop}
Let $\Delta$ be a region.  Then $\kappa(\Delta)\le0$.
\end{prop}

\begin{pf}
Since no angle is greater than $\pi$, the result follows immediately for regions of non-positive Euler characteristic,
so it suffices to consider regions which are topological discs.

If $\Delta$ is a boundary region which is a disc, then it has at least $4$ corners with angle $\pi/2$ (and hence $\kappa(\Delta)\le0$), except in one of the following two possible cases.

\begin{enumerate}
\item $\Delta=D^2$.  In this case $\mathcal{P}$ is empty, and $Z$ is the empty word, contrary to hypothesis.
\item $\partial\Delta$ consists of a single edge $e$
together with a single arc $\gamma\subset\partial D^2$.  In this case the label on $\gamma$ is trivial, so $e$ and $\Delta$ can be removed from $\mathcal{P}$, contrary to the assumption of minimality.
\end{enumerate}

By definition, if $\Delta$ is a triangle, then $\theta(c)=\pi/3$ for every corner $c$
of $\Delta$, and so $\kappa(\Delta)=0$. 

 Hence we may assume that $\Delta$ is an interior $k$-gonal region for some $k\ge 4$.  By definition
$\theta(c)\le 2\pi/3$ for each corner $c$ of
$\Delta$, and so if $k\ge 6$ then $\kappa(\Delta)\le0$.  Hence we are reduced
to the case where $k\in\{4,5\}$.

In any $5$-gonal interior region, the labels on the corners must consist of $4$ labels $y$ and one $y^2$ (or {\em vice versa}),
and so a bridge-move is possible across such a region (in either of two ways.

A $5$-gon with $\kappa>0$ must have all its vertices of degree $5$.  However, any edge joining interior vertices
of degree $5$ represents two or more arcs.  So a bridge-move
across such a region creates in its place a triangular region and a $4$-gonal region, leaving the rest of the picture unchanged.

It follows from our assumption 
of maximality of the number of triangles that no
pentagonal region has $\kappa>0$.  Thus we are reduced to the case $k=4$.  In particular, precisely two of the corner labels in $\Delta$ are $y$ and two are $y^2$.

Note also that $\kappa(\Delta)\le0$ unless
$\Delta$ has at least three corners belonging to vertices of degree $5$.  By Lemma \ref{ind5} the corner of
any such vertex in $\Delta$ has label one of
$0,1,2,3,4,5,7,8,9,A$.

We complete the proof by considering separate cases.

\medskip\noindent{\bf Case 1. $\Delta$ has a vertex
of degree $5$ with corner $5$ in $\Delta$.}

\medskip
We assume without loss of generality that this
vertex $v$ has positive orientation.

The anti-clockwise edge of $\Delta$ at $v$ carries $[6..9]$ which has a unique match $\overline{[1..4]}$,
so the neighbouring vertex at the other end of this edge has
positive orientation and corner $5$ (hence label $y$) in
$\Delta$.

The clockwise edge from $v$ in $\Delta$ carries
$[n..4]$ with $n\in\{1,2\}$. The possible matches
are $\overline{[6..9]}$ or $\overline{[6..8]}$ (if the next vertex is
positively oriented); and $[2..5]$ (if it is negatively oriented).  In either case the next vertex has corner in
$\Delta$ labelled $y$ (positively oriented corner $0$ or negatively oriented corner $6$).  But then the $4$-gon $\Delta$
has at least three corners labelled $y$, which is impossible (Figure \ref{c1}).

This contradiction shows that the result holds whenever $\Delta$ has a vertex of degree $5$ with $5$-corner in $\Delta$. 

\begin{center}
\begin{figure}
\includegraphics[scale=0.4]{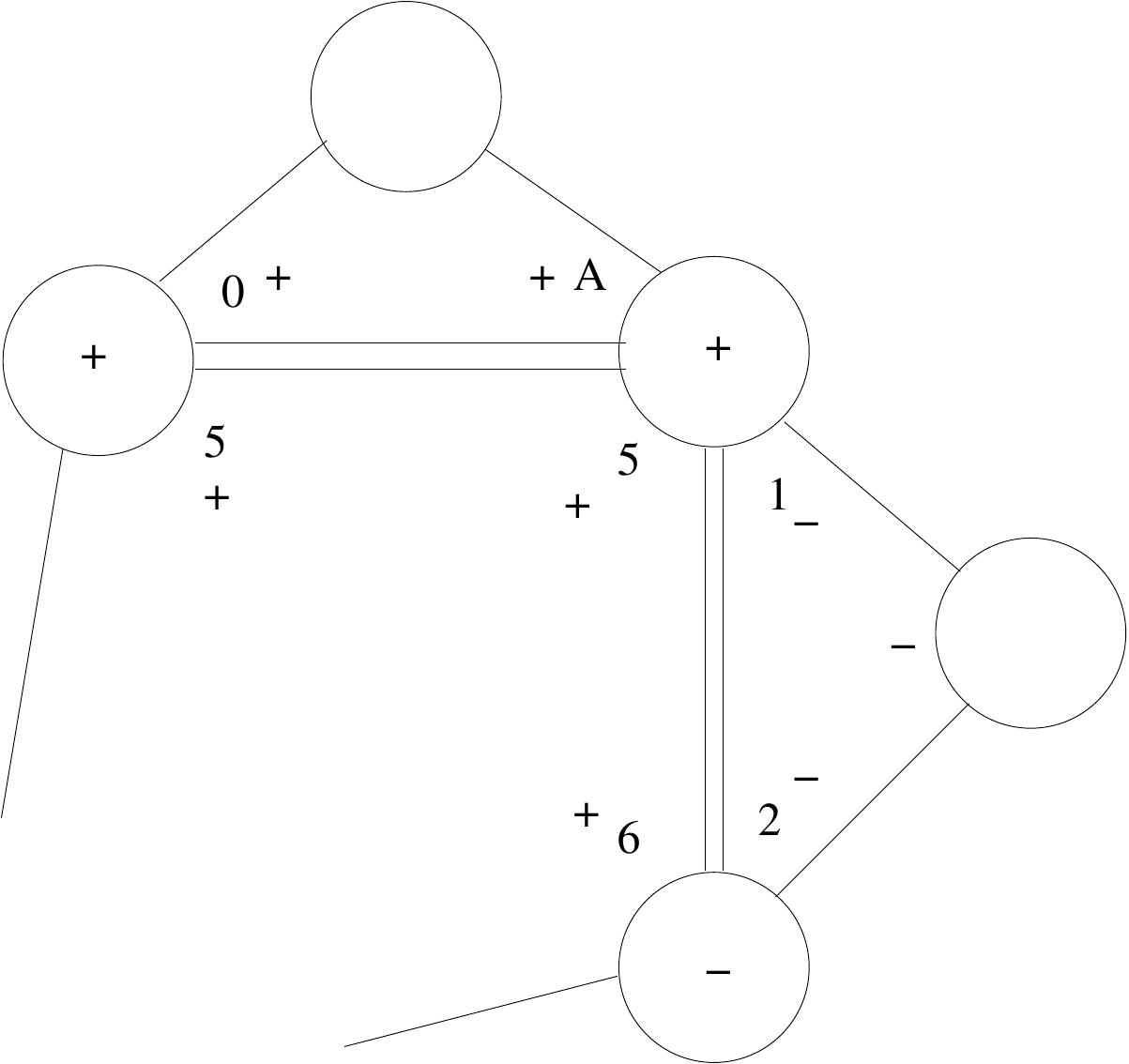}
\caption{Case 1}
\label{c1}
\end{figure}
\end{center}

\medskip\noindent{\bf Case 2. $\Delta$ has a vertex
of degree $5$ with corner $A$ in $\Delta$.}

\medskip
As in Case 1, we assume without loss of generality that this
vertex $v$ has positive orientation.

The clockwise edge of $\Delta$ at $v$ carries $[6..9]$ which has a unique match $\overline{[1..4]}$,
so the neighbouring vertex at the other end of this edge has
positive orientation and corner $0$ (hence label $y$) in
$\Delta$.

The anti-clockwise edge from $v$ in $\Delta$ carries
$[B..n]$ with $n\in\{1,2,3\}$.  In each case there is a unique match $\overline{[(1-n)..3]}$. Hence
the neighbouring vertex at the other end of this edge has
positive orientation and corner $4$ (hence label $y$) in
$\Delta$.  Thus $\Delta$ has three corners labelled $y$, a contradiction (Figure \ref{c2}).

This contradiction shows that the result holds whenever $\Delta$ has a vertex of degree $5$ with $A$-corner in $\Delta$. 

\begin{center}
\begin{figure}
\includegraphics[scale=0.4]{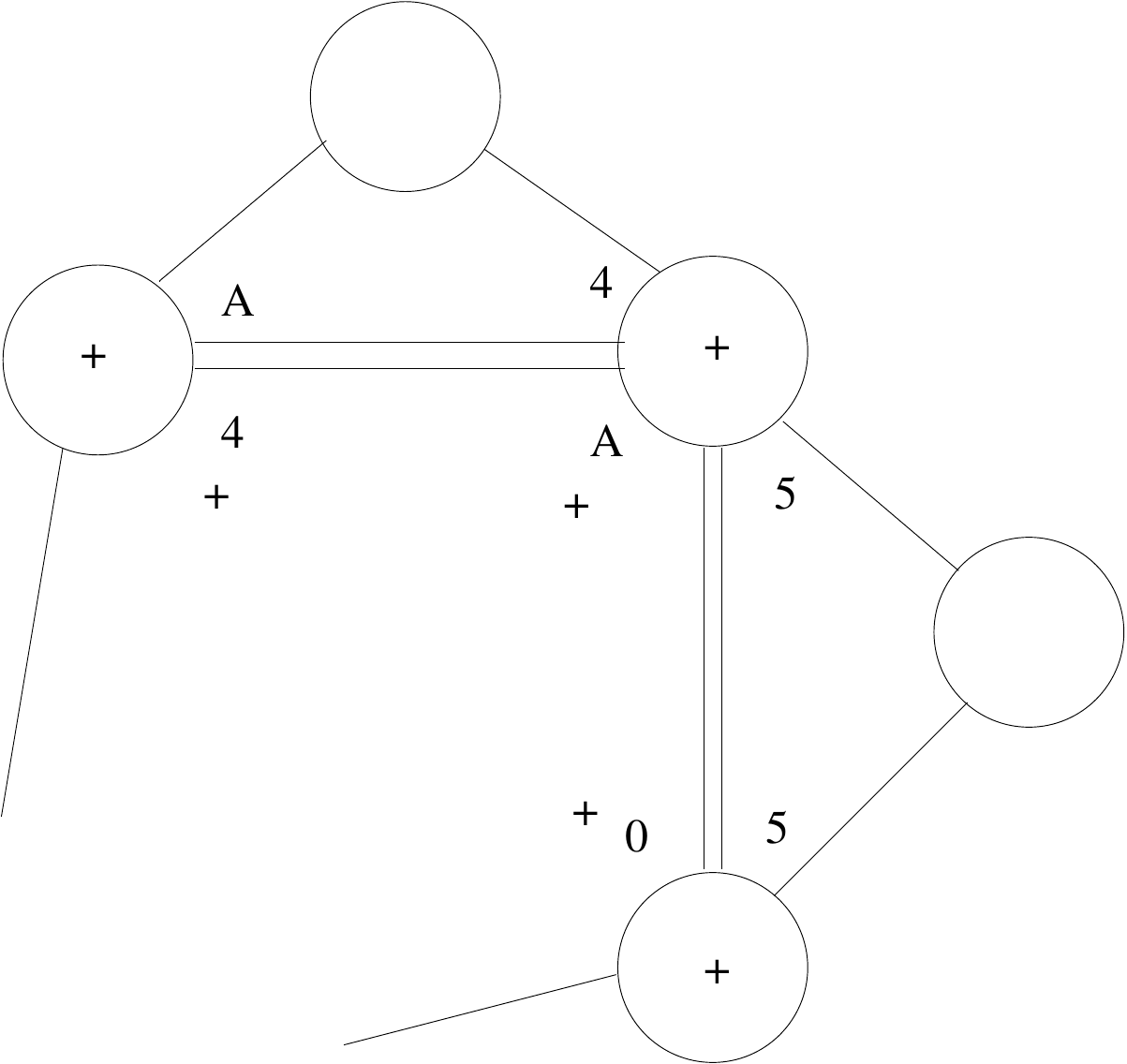}
\caption{Case 2}
\label{c2}
\end{figure}
\end{center}

\medskip\noindent{\bf Case 3. $\Delta$ has a vertex
of degree $5$ with corner $0$ in $\Delta$.}

\medskip
As in Case 1, we assume without loss of generality that this
vertex $v$ has positive orientation.

The anti-clockwise edge of $\Delta$ at $v$ carries $[1..4]$ which has a unique match $\overline{[6..9]}$,
so the neighbouring vertex $v_0$ at the other end of this edge has
positive orientation and corner $A$ (hence label $y$) in
$\Delta$.  By Case 2, we may assume that this vertex has degree $>5$, so the
result follows unless the remaining two vertices
of $\Delta$ have degree $5$ and corner label $y^2$
in $\Delta$.

The clockwise edge of $\Delta$ at $v$ carries $[n..C]$
for $n\in\{8,9,A\}$. Since $v$ and its neighbour $v_1$ have
opposite corner labels in $\Delta$, the match must extend
to a match for $[n..0]$.  The only possibility is
$\overline{[9..(8-n)]}$, so the vertex is positively oriented
with corner $8$ in $\Delta$ (Figure \ref{c3}).

\begin{center}
\begin{figure}
\includegraphics[scale=0.4]{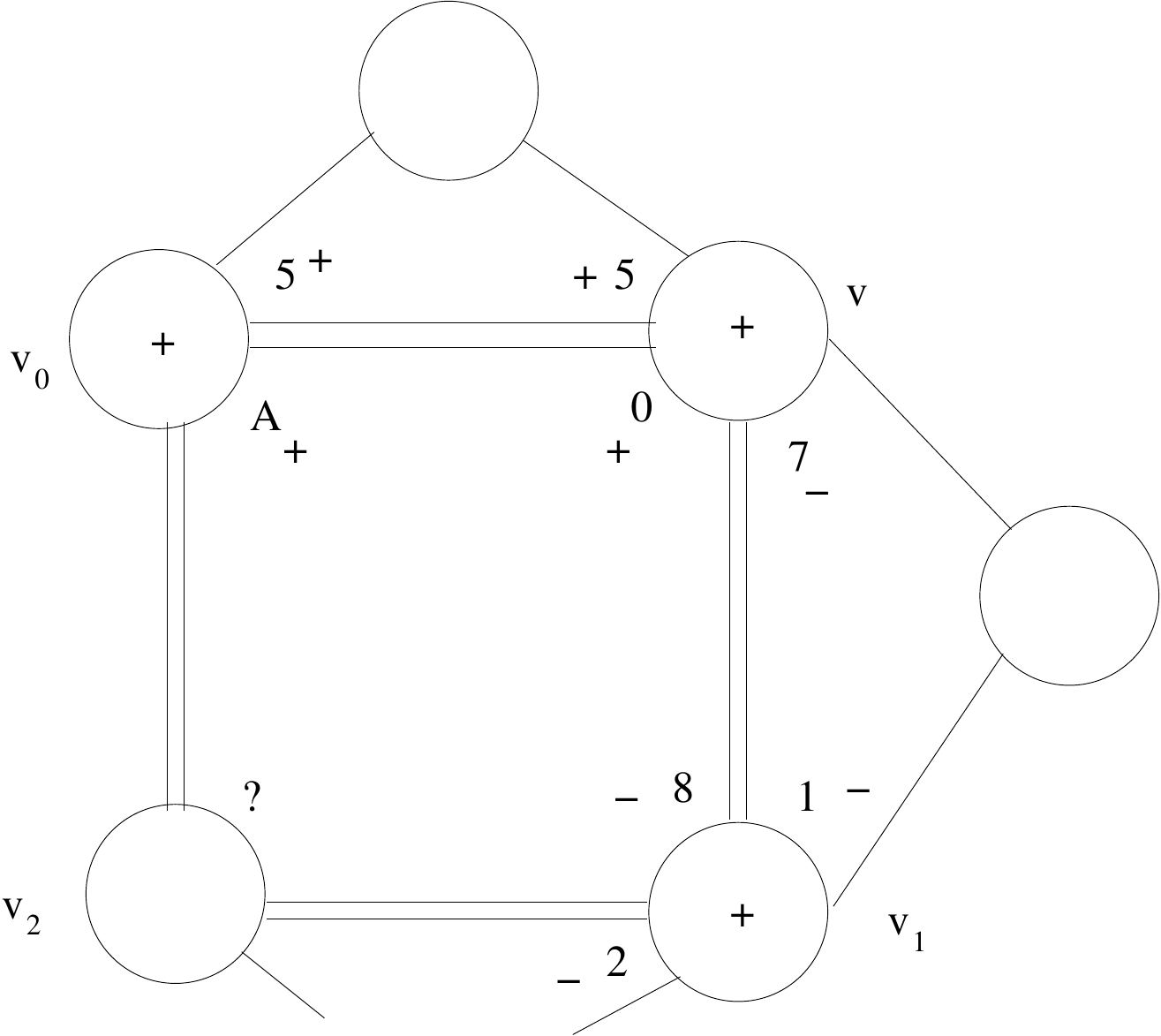}
\caption{Case 3}
\label{c3}
\end{figure}
\end{center}

Now consider the edge of $\Delta$ connecting $v_1$ to the fourth vertex $v_2$ of $\Delta$.  This edge carries
$[n..7]$ at $v_1$, for $n\in\{3,4,5\}$.  The match at $v_2$
cannot extend to a match for $[n..8]$, since the corners 
in $\Delta$ of $v_1$ and $v_2$ have the same label.  The only
possibilities are $[(n+5)..C]$ and $[\overline{[9..(4-n)]}$.

Thus $v_2$ is either positively oriented with corner $8$ in 
$\Delta$, or negatively oriented with corner $0$ in $\Delta$.
In the first case, we may apply the same argument to the
edge joining $v_2$ to $v_0$ and deduce that $v_0$ has corner $0$ or $8$ in $\Delta$.  But we have already seen that $v_0$ has corner $A$ in $\Delta$, a contradiction.  In the second case, we may apply the whole analysis to the corner $v_2$
rather than $v$, and deduce that $v_0$ and $v_1$ are negatively oriented.  Again this is a contradiction.

This contradiction shows that the result holds whenever $\Delta$ has a vertex of degree $5$ with $5$-corner in $\Delta$. 

\medskip\noindent{\bf Case 4. $\Delta$ has a vertex
of degree $5$ with corner $1$ in $\Delta$.}

\medskip
As in Case 1, we assume without loss of generality that this
vertex $v$ has positive orientation.

The clockwise edge of $\Delta$ at $v$ carries $[n..0]$ with
$n\in\{8,9,A\}$.  There is a unique match $\overline{[8..(8-n)]}$. Hence the neighbouring vertex $v_1$ at the other end of this edge has
positive orientation and corner $7$ in
$\Delta$. 
In particular the corners of $v,v_1$ in $\Delta$ both have
the same label $y^2$.  Hence we may assume that the other
two corner labels in $\Delta$ are $y$.

If any edge of $\Delta$ consists only of a single arc, then
the vertices at either end of this edge each have degree $>5$, and so
$\kappa(\Delta)\le0$.  Hence we may assume that each edge of
$\Delta$ consists of two or more arcs.  Now carry out a bridge move across $\Delta$, replacing $\Delta$ by a
$4$-gonal region $\Delta'$ with corner $0$ at $v$ (Figure \ref{b4}).  This bridge move has not changed the degree of any vertex of $\Delta$, so $\kappa(\Delta)=\kappa(\Delta')$.  But
$\kappa(\Delta')\le0$ by Case 3.

\begin{center}
\begin{figure}
\includegraphics[scale=0.35]{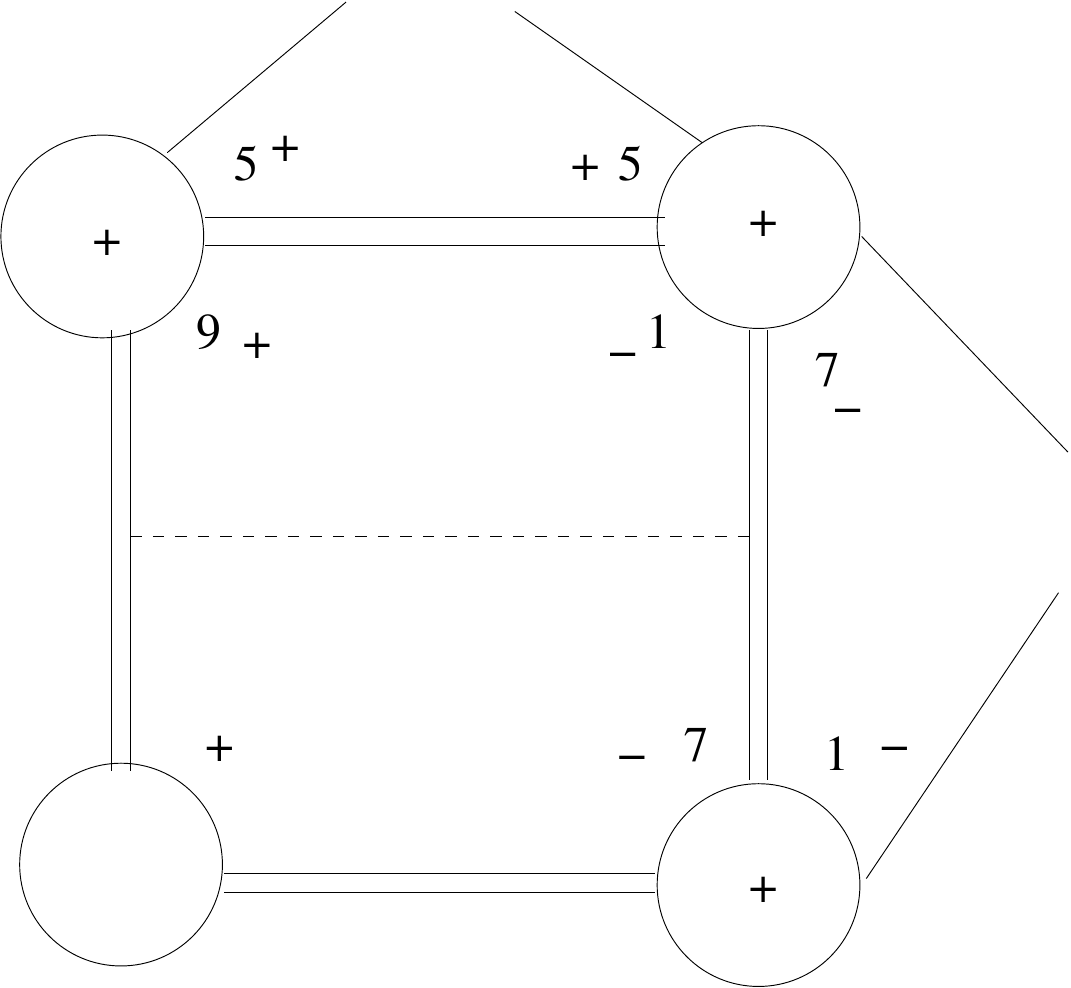}
\qquad \qquad
\includegraphics[scale=0.35]{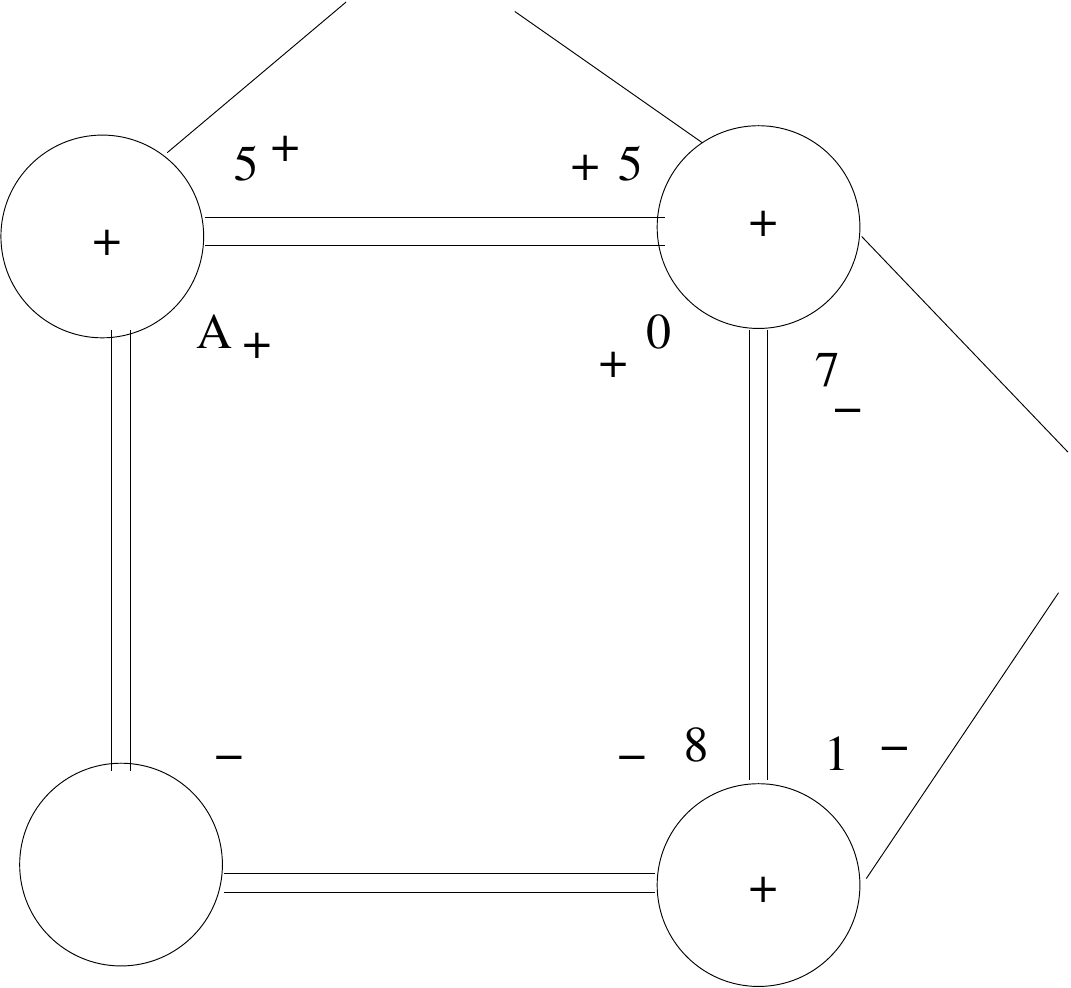}
\caption{Bridge move in Case 4}
\label{b4}
\end{figure}
\end{center}

\medskip\noindent{\bf Case 5. $\Delta$ has a vertex
of degree $5$ with corner $4$ in $\Delta$.}

\medskip
As in Case 1, we assume without loss of generality that this
vertex $v$ has positive orientation.

The clockwise edge of $\Delta$ at $v$ carries $[B..3]$.  There is a unique match $\overline{[B..3]}$. Hence the neighbouring vertex $v_1$ at the other end of this edge has
positive orientation and corner $A$ in
$\Delta$.   In particular we may assume by Case 2 that
$v_1$ has degree $>5$.  Moreover, the corners of $\Delta$
at $v,v_1$ are both labelled $y$, so we may assume that the other two corners are
of degree 5 and labelled $y^2$.

The anti-clockwise edge of $\Delta$ at $v$ carries $[5..n]$
with $n\in\{6,7,8\}$. We may assume that this edge separates $\Delta$ from a triangular region, for otherwise each
of its incident vertices $v,v_0$ has at least $2$ non-triangular corners.  This, together with the fact that $v_1$
has degree $>5$, would imply that $\kappa(\Delta)\le0$.

Since the labels of the corners of $\Delta$ at $v,v_0$ are opposite, the match for $[5..n]$ at $v_0$ must extend to a match for $[4..n]$, but not to a match for $[5..(n+1)]$.  There
are precisely three possible matches: $[A..C]$ or $\overline{[9..B]}$ with $n=7$, and $\overline{[3..6]}$ with $n=8$.  Of these, we may rule out $\overline{[9..B]}$, since it would mean that $v_0$ had a corner $C$ and so degree at least $6$.  Since neither $2$ nor $9$ can be a corner of a triangular region at an interior vertex of degree $5$, we are
left only with the possible match $[A..C]$.  Thus $v_0$ is
negatively oriented, with corner $9$ in $\Delta$, and label $y^2$.

\begin{center}
\begin{figure}
\includegraphics[scale=0.4]{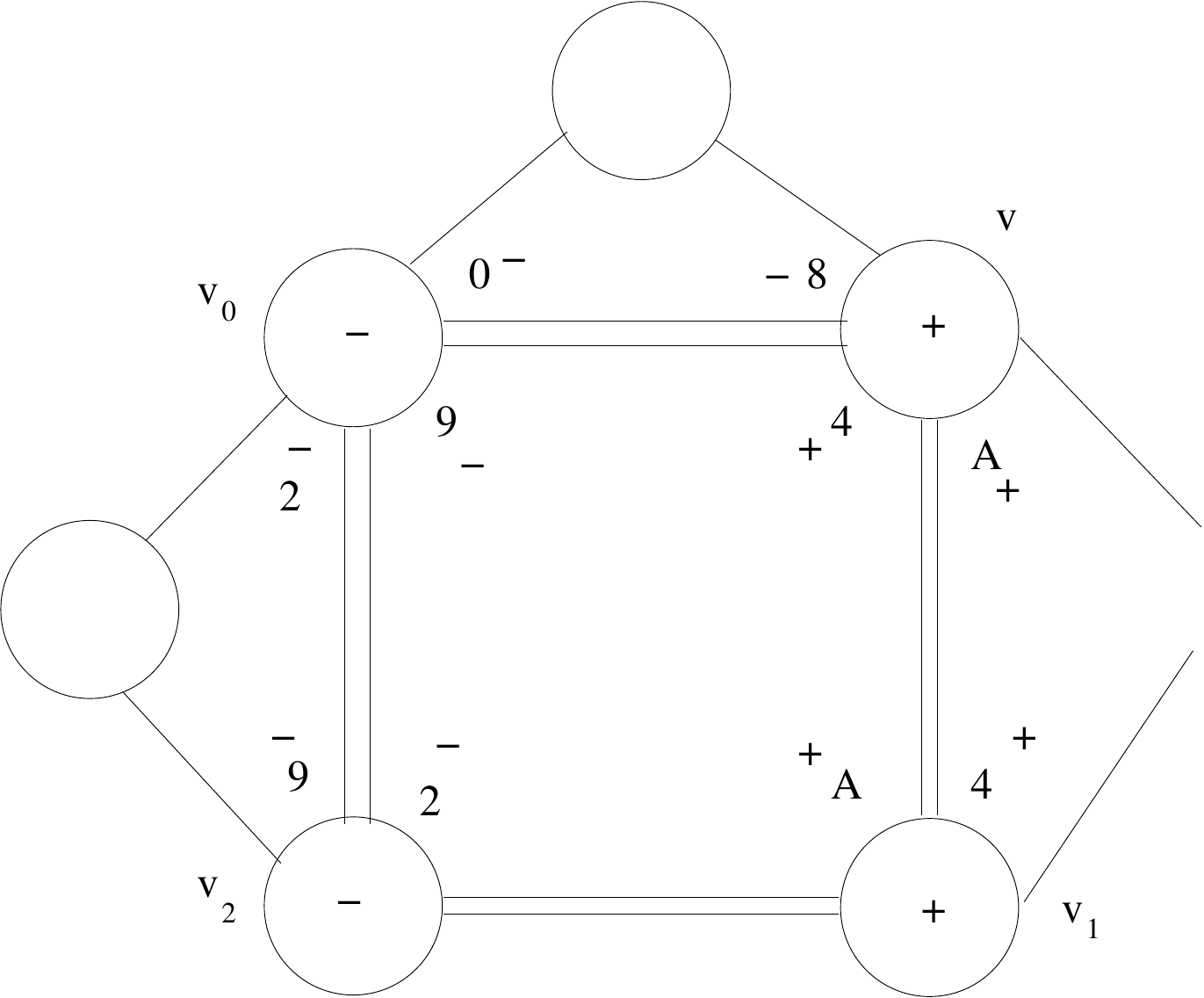}
\caption{Case 5}
\label{c5}
\end{figure}
\end{center}

Now consider the edge from $v_0$ to the fourth vertex $v_2$ of $\Delta$.  Again we may assume that this edge separates $\Delta$ from a triangular region $T$. The corners at $v_0,v_2$ in $T$ must be equal, as are the corners
at $v_0,v_2$ in $\Delta$.  This edge carries $[n..8]$ at $v_0$, and the match at $v_2$ must be maximal.  The only possible match is $\overline{[3..8]}$ with $n=3$. In particular $v_0$ has corner $2$ in $T$.  But an interior vertex of degree $5$ cannot have a triangular corner $2$, so once again we have a contradiction (Figure \ref{c5}).

\medskip\noindent{\bf Case 6. $\Delta$ has a vertex
of degree $5$ with corner $c\in\{2,3\}$ in $\Delta$.}

\medskip
As in Case 1, we assume without loss of generality that this
vertex $v$ has positive orientation.

The clockwise edge of $\Delta$ at $v$ carries $[B..(c-1)]$.  There is a unique match $\overline{[(2-c)..3]}$. Hence the neighbouring vertex $v_1$ at the other end of this edge has
positive orientation and corner $B$ or $C$ in
$\Delta$, and hence label $y^2$.   In particular 
$v_1$ has degree $>5$. Thus we may assume that the other two corners have degree $5$.

Arguing as in Case 4, we may perform a bridge move across $\Delta$ without changing the degrees of the incident vertices
or changing the rest of the picture (Figure \ref{b6}).  The resulting $4$-gon
$\Delta'$ has a corner $c+1$.  If $c=3$ then by Case 5 we have $\kappa(\Delta')\le0$.  Hence also $\kappa(\Delta)=\kappa(\Delta')\le0$.
If $c=2$ then $\Delta'$ has a vertex of degree $5$ with corner $3$ in $\Delta'$.  As we have just seen, this implies that $\kappa(\Delta')\le0$, and hence 
by the same argument again we have $\kappa(\Delta)\le0$.

\begin{center}
\begin{figure}
\includegraphics[scale=0.35]{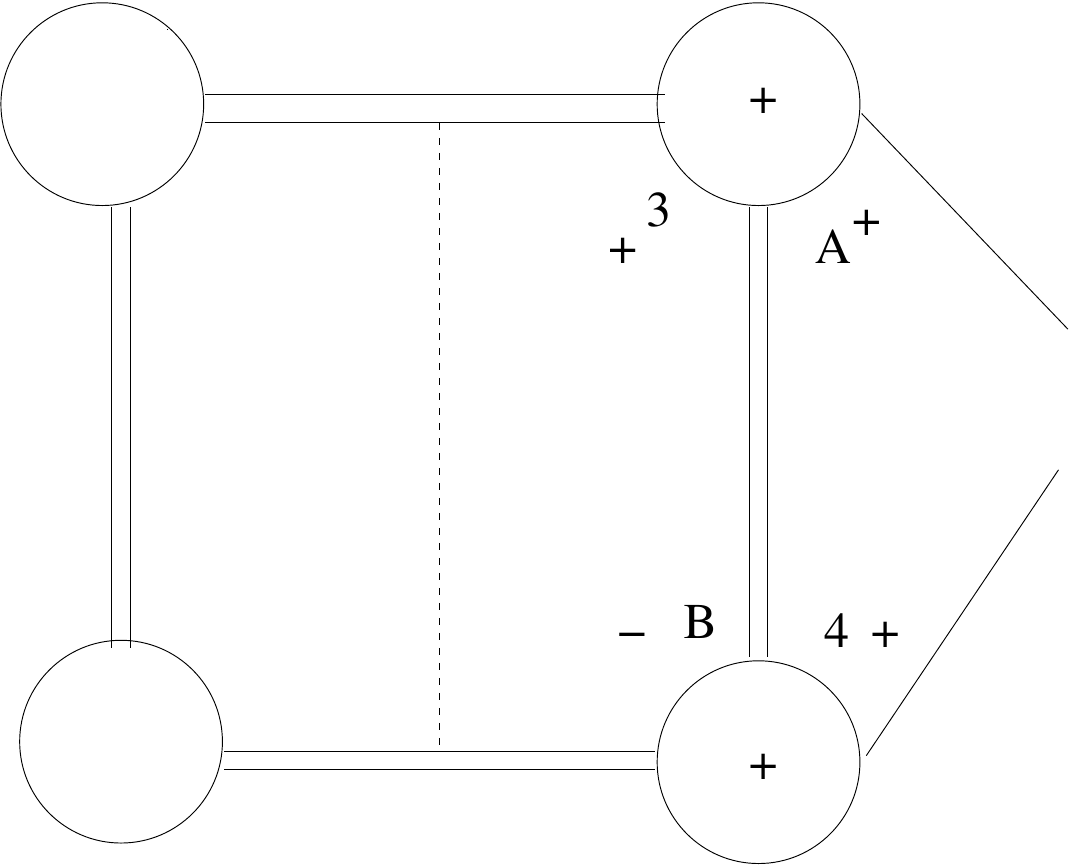}
\qquad \qquad
\includegraphics[scale=0.35]{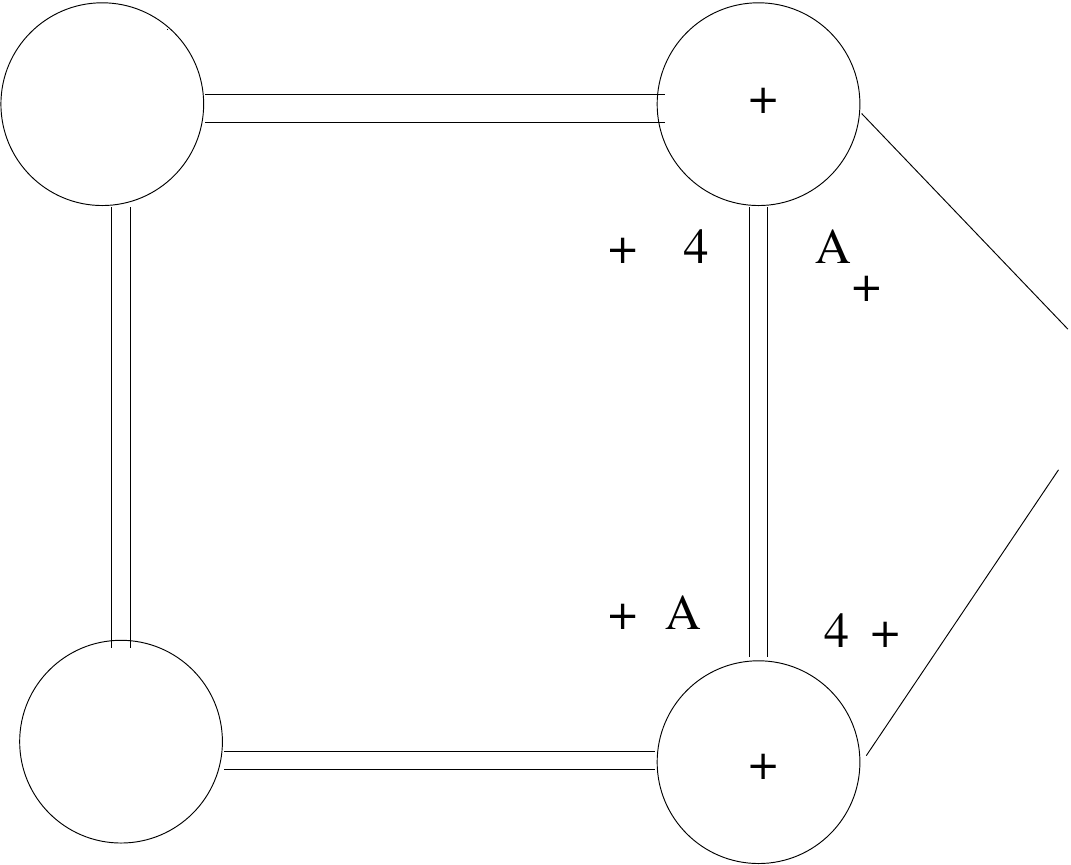}
\caption{Bridge move in Case 6}
\label{b6}
\end{figure}
\end{center}

\medskip\noindent{\bf Conclusion}

\medskip
Using the results of Cases 1-6, we may assume that $\Delta$ has three consecutive vertices
of degree $5$, each with corner $7$, $8$ or $9$ in $\Delta$.
We assume without loss of generality that the middle 
vertex $v$ of this sequence has positive orientation,
and corner $c\in\{7,8,9\}$.

Suppose first that $c=9$.  Then the clockwise edge of $\Delta$
from $v$ carries $[n..8]$ for $n\in\{3,4,5\}$.  There is a unique match $\overline{[3..(11-n)]}$ for this, so the
vertex $v_1$ at the other end of this edge has corner $2$ in $\Delta$, contrary to hypothesis.  Hence $c\in\{7,8\}$.

The anti-clockwise edge of $\Delta$ at $v$ carries $[(c+1)..n]$ with $n\in\{C,0\}$. If $n=0$ then there is a unique match $\overline{[8..(7-c)]}$. Hence the neighbouring vertex $v_0$ at the other end of this edge has
positive orientation and corner $0$, $1$ or $C$ in
$\Delta$, contrary to assumption.  So we may assume that $n=C$.  If $c=7$ then again there is a unique match
$\overline{[9..0]}$; in this case $v_0$ has corner $1$ in $\Delta$,
again giving a contradiction.  So we may assume that $c=8$.  
Possible matches for $[9..C]$ are $\overline{[9..C]}$,
$[4..7]$ and $\overline{[4..7]}$.  The corner of
$v_0$ in $\Delta$ is then $0$, $3$, or $8$
respectively.

By hypothesis, $v_0$ has corner $7$, $8$ or $9$ in $\Delta$,
so the only possible match is $\overline{[4..7]}$.  In particular $v_0$ is positively oriented, with corner  $8$ in $\Delta$ and hence label
$y^2$. The corner of $\Delta$ at $v$ also has label $y^2$, so
we may assume that the other two corners of $\Delta$ have label $y$.

As in previous cases, we may assume that each edge of $\Delta$ corresponds to two or more arcs, so that a bridge
move across $\Delta$ has the effect of replacing $\Delta$ by another $4$-gonal region $\Delta'$, without changing the degrees of the vertices of $\Delta$ (Figure \ref{b7}).  Such a bridge-move
also has the effect that the corners of $v,v_0$ in $\Delta$ 
change from $8$ to $9,7$ respectively.  Hence $\kappa(\Delta)=\kappa(\Delta')\le0$, as required.

\begin{center}
\begin{figure}
\includegraphics[scale=0.35]{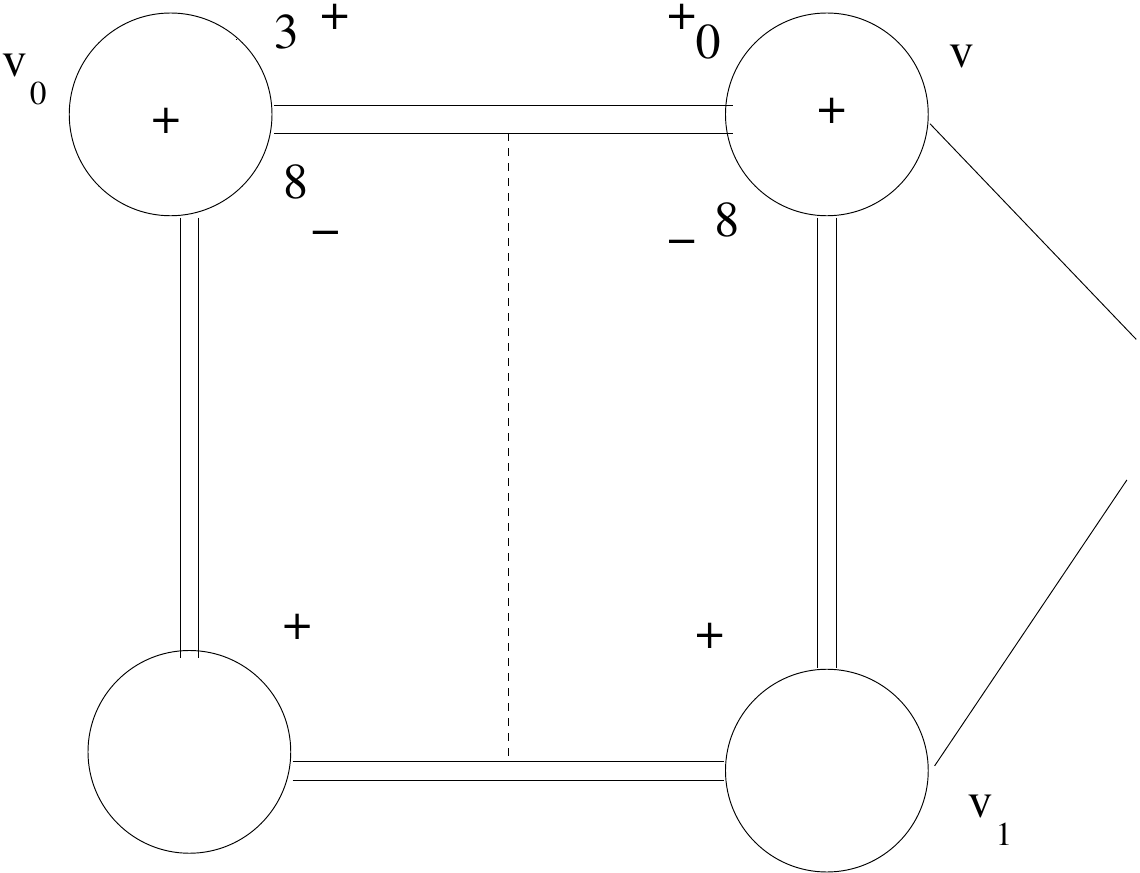}
\qquad \qquad
\includegraphics[scale=0.35]{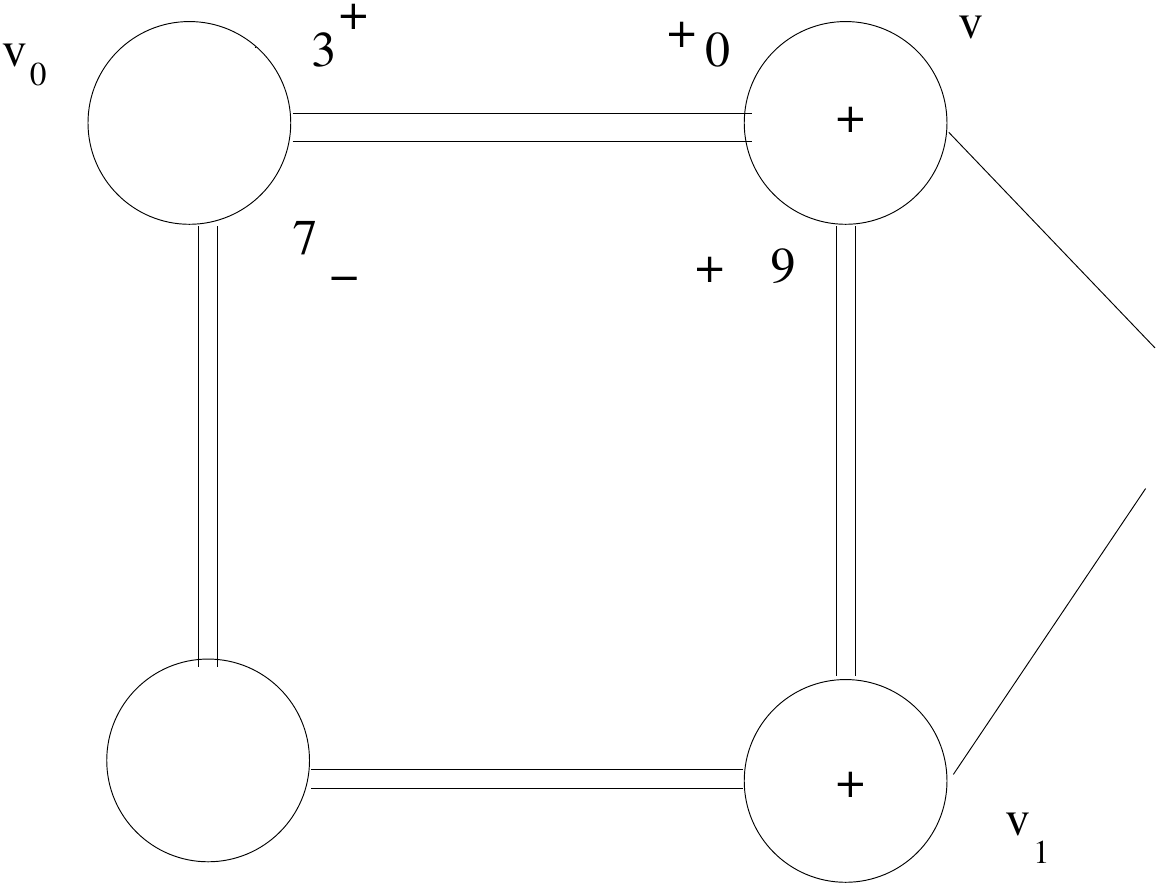}
\caption{Bridge move to complete proof}
\label{b7}
\end{figure}
\end{center}

This completes the proof.
\end{pf}

As a consequence of the above result, any positive curvature in our picture is concentrated at interior vertices of
degree $5$ with $5$ triangular corners. 

Hence there exist interior vertices with $\kappa>0$.
To complete the proof we distribute positive curvature
from these vertices to negatively curved neighbours.

Recall from Theorem \ref{posvertex} that any such interior vertex $v$ has corners
$5,A,4,7,1$, and that the pieces
$[2..4]$ and $[6..9]$ are matched at the neighbouring vertices $u_0,u_1$ by $[3..5]$ and $[1..4]$ respectively (Figure \ref{P4}).

\begin{center}
\begin{figure}
\includegraphics[scale=0.35]{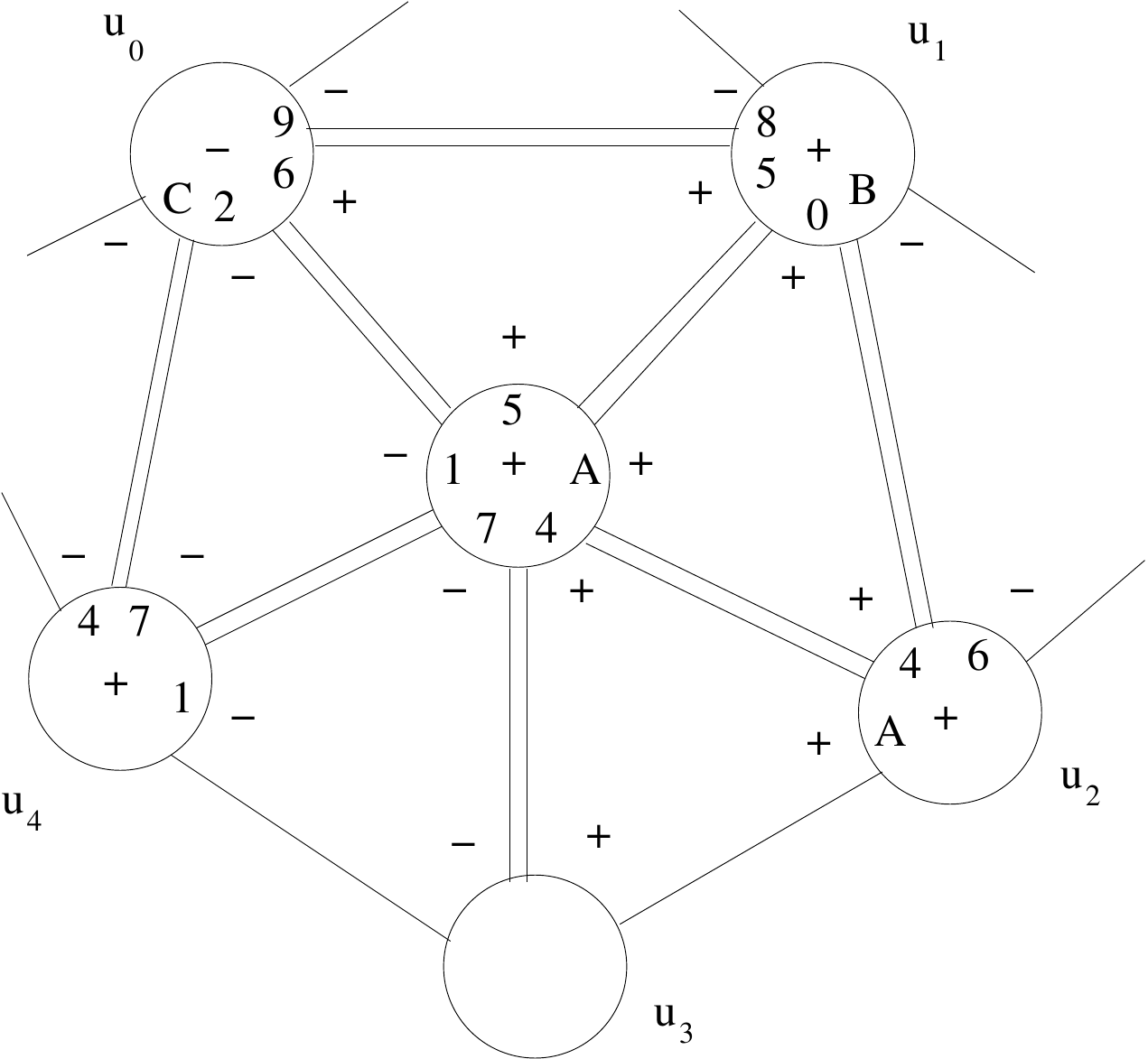}
\caption{Neighbourhood of a positively curved vertex}
\label{P4}
\end{figure}
\end{center}

We say that these two vertices $u_0,u_1$ are {\em associated}
to $v$.  Note that each of $u_0,u_1$ is connected to $v$
by an edge carrying a word that contains $[3..4]$.  In particular, no vertex is associated to more than $2$ such interior vertices with $\kappa>0$.  It also follows that no
such vertex can itself be positively curved.

We define the {\em amended} curvature function on vertices
by
\begin{itemize}
\item $\kappa'(v)=\kappa(v)+k\pi/6$ if $u$ is associated to
$k>0$ interior vertices of positive ($\kappa$-)curvature;
\item $\kappa'(v)=0$ if $v$ is an interior vertex with $\kappa(v)>0$;
\item $\kappa'(v)=\kappa(v)$ otherwise.
\end{itemize}

In other words, we distribute $\pi/6$ of positive curvature
from any positively curved vertex to each of its two associated vertices.  There is no adjustment to the curvature of regions, so the total curvature remains unchanged.
We have already seen that there are no positively curved regions, so it remains to check that there are no vertices with positive amended curvature.

\begin{lemma}\label{last}
For every vertex $u$ of $\mathcal{P}$, $\kappa'(u)\le0$.
\end{lemma}

\begin{pf}
Suppose first that $u$ is an interior vertex of $\mathcal{P}$.
If $\kappa(u)>0$ then by Theorem \ref{posvertex} we know that
$u$ has degree $5$ with all incident corners triangular, so
$\kappa(u)=\pi/3$.  By definition, $\kappa'(u)=0$.

So we may assume that $\kappa(u)\le0$.

If $u$ is not associated to any interior vertex with
positive $\kappa$-curvature, then by definition $\kappa'(u)=\kappa(u)\le0$.

If $u$ is associated to an interior vertex $v$ of positive curvature, then $v$ has a configuration as in Figure \ref{P4} above,
and $u$ is one of the neighbouring vertices $u_0,u_1$.

Note that the edges connecting $u$ to $u_2,u_4$
carry $[B..3]$ and $[8..0]$, which have unique matches
$\overline{[B..3]}$ and $\overline{[8..0]}$ respectively.
Hence $u_2,u_4$ are positively oriented.
Hence also the corners of the triangle $uu_1u_2$ at $u_1$ and $u_2$
are $0$ and $4$ respectively, while those of $uu_4u_0$ at $u_4$ and $u_0$ are $7$ and $2$ respectively.

The labels following corner $4$ at $u_2$, in clockwise order,
are $y,y^2,y^2,\dots$, while those following $0$ at $u_1$ in anti-clockwise order are $y^2,y^2,y,\dots$.  It follows that the edge from $u_1$ to $u_2$ has at most two arcs, so the next corner of $u_2$ anticlockwise from the corner $0$ in $uu_1u_2$ is $B$ or $C$.

By a similar analysis, we see that the edge $u_4u_0$ has at most $3$ arcs, so that the next corner of $u_4$ clockwise from the corner $2$ in $uu_4u_0$ is $C$, $0$ or $1$.

A similar analysis again shows that the edge $u_0u_1$ has at most $3$ arcs, so the next corner of $u_0$ anti-clockwise from the
$6$ corner in $uu_0u_1$ is $7$, $8$ or $9$, while the
next corner of $u_1$ clockwise from the
$5$ corner in $uu_0u_1$ is $6$, $7$ or $8$.

Finally, since $[A..1]$, $[2..5]$ and $[6..A]$ are non-pieces,
it follows that each of $u_0,u_1$ have at least three incident corners in addition to the four in Figure \ref{P4}.  In other words, each has degree at least $7$, and hence $\kappa(u)\le-\pi/3$.  Since $u$ is associated to at most two
interior vertices of positive curvature, $\kappa'(u)\le\kappa(u)+2\pi/6\le0$, as required.

\medskip
Now suppose that $u$ is a boundary vertex.  If $u$ is connected to the boundary by more than one parallel class of arcs, then it has at least four corners of angle $\pi/2$, and hence $\kappa(u)\le0$.  If in addition it is associated to a
positively curved interior vertex, then it has also at least two interior corners, with angles $\pi/3$.  Hence $\kappa(u)\le-2\pi/3$ so $\kappa'(u)\le-2\pi/3+2\pi/6<0$.

If $u$ is not connected to the boundary, then it has degree at least $5$ and at least one corner with angle $\pi$, so
$\kappa(u)\le-\pi/3$ and again $\kappa'(u)\le0$.

Assume therefore that $u$ is connected to the boundary by a single parallel class of arcs.  The word $U$ carried by that class has to be a cyclic subword of $W^2$, and must match a
cyclic subword of $Z$

The maximal such matches are
\begin{itemize}
\item $[(xy)^4(xy^{-1})^3x]^{\pm 1}$,
\item $[(xy^{-1})^3(xy)^2x]^{\pm 1}$,
\item $[xyxy^{-1}x(yx)^4]^{\pm 1}$, and 
\item $[(xy^{-1})^2xyxy^{-1}x]^{\pm 1}$.
\end{itemize}

In particular $U$ is either a piece or a product of two pieces.  We may
extend $\mathcal{P}$ to a new reduced picture $\widehat{\mathcal{P}}$,
in which $u$ is no longer attached to the boundary by arcs, by
joining either one (if $U$ is a piece) or two new vertices to $u$
where it joins $\partial\mathcal{P}$. 

If $u$ is associated to a positively curved interior vertex $v$ of
$\mathcal P$, then $v$ is also a positively curved interior vertex of
$\widehat{\mathcal P}$ to which $u$ is associated.  It follows that $u$ has degree at least $7$ in $\widehat{\mathcal P}$, and hence at least $5$ in $\mathcal{P}$.  It therefore has
at least $4$ interior corners, and two boundary corners,
so $\kappa(u)\le 2\pi -2\pi/2 - 4\pi/3=-\pi/3$, and hence $\kappa'(u)\le0$ as required.

Otherwise  $\kappa'(u)=\kappa(u)$ so it suffices to show that
$\kappa(u)\le0$.

 By Lemma \ref{ind5}, $u$ has index at least
$5$ in $\widehat{\mathcal{P}}$, so at least $4$ in $\mathcal{P}$,
except possibly in the case where $U$ is not a piece and the index of
$u$ in $\widehat{\mathcal{P}}$ is precisely $5$.  In this case $U$
is one of $[(xy)^4(xy^{-1})^kx]^{\pm 1}$ (with $0\le k\le 3$),
$[(xy^{-1})^3(xy)^2x]^{\pm 1}$, $[xyxy^{-1}x(yx)^4]^{\pm 1}$
or $[xy^{-1}x(yx)^4]^{\pm 1}$.  In the first case, $u$ would be
a vertex of index $5$ in $\widehat{\mathcal{P}}$, not joined to the boundary,
with an incident corner $C$. In the second case $u$ would have an incident corner $B$.
In the last two cases
$u$ would have an incident corner $6$.  In all cases we would have a contradiction to Lemma \ref{ind5}.
Hence $u$ has index at least $4$ in $\mathcal{P}$
in all cases.

But then $u$ has at least three corners of angle $\ge\pi/3$
together with two boundary corners of angle $\pi/2$, so
$\kappa'(u)=\kappa(u)\le0$ as claimed.
\end{pf}

This concludes the proof of Theorem \ref{free13}.

\end{document}